\documentclass[11pt,twoside]{article}
\usepackage{amsmath,amssymb,epsfig,subfigure}
\newtheorem{proposition}{Proposition}[section]
\newtheorem{theorem}[proposition]{Theorem}
\newtheorem{lemma}[proposition]{Lemma}

\newtheorem{algorithm}[proposition]{Algorithm}

\newenvironment{proof}{{\noindent \bf Proof:}}{\hfill $\fbox{}$ \vspace*{5mm}}

\renewcommand{\v}[1]{\ensuremath{\mathbf{#1}}}

\newcommand{\bn}{{\bf n}}

\newcommand{\bz}{{\bf z}}

\newcommand{\ch}{{\cal H}}

\newcommand{\cu}{{\cal U}}
\newcommand{\cz}{{\cal Z}}
\newcommand{\R}{\mathbb{R}}

\newcommand{\Rnt}{\mathbb{R}^{n^2}}
\newcommand{\Rntnt}{\mathbb{R}^{n^2\times n^2}}
\newcommand{\Rtnt}{\mathbb{R}^{2n^2}}

\newcommand{\BE}{\begin{equation}}
\newcommand{\EE}{\end{equation}}

\begin{document}
\title{A Fast Alternating Minimization Algorithm for Total Variation Deblurring Without Boundary Artifacts}

\author{Zheng-Jian Bai\thanks{School of Mathematical Sciences, Xiamen
University, Xiamen 361005, People's Republic of China,
Dipartimento di Scienza e alta Tecnologia, Universit{\`{a}}
dell'Insubria - Sede di Como, Via Valleggio 11, 22100 Como, Italy,
\textbf{E-mail:} zjbai@xmu.edu.cn. The research of this author was
partially supported by the National Natural Science Foundation of
China grant 11271308,  the Natural Science Foundation of Fujian
Province of China for Distinguished Young Scholars (No.
2010J06002), NCET, and Internationalization Grant of U. Insubria
2008.} \and Daniele Cassani
\thanks{Dipartimento di  Scienza e alta Tecnologia,
Universit{\`{a}} dell'Insubria - Sede di Como, Via Valleggio 11,
22100 Como, Italy, \textbf{E-mail:}
$\{$daniele.cassani,marco.donatelli,stefano.serrac$\}$@uninsubria.it
The work of Marco Donatelli and Stefano Serra-Capizzano was partially supported by MIUR, grant number 20083KLJEZ.
}
\and Marco Donatelli$\mbox{ }^{\dagger}$
\and Stefano Serra-Capizzano$\mbox{ }^{\dagger}$
}

\maketitle
\begin{abstract}
Recently, a fast alternating minimization algorithm for total
variation image deblurring (FTVd) has been presented by Wang, Yang, Yin,
and Zhang [{\em SIAM J. Imaging Sci.}, 1 (2008), pp. 248--272].
The method in a nutshell consists of a discrete Fourier transform-based alternating
minimization algorithm with periodic boundary conditions and in
which two fast Fourier transforms (FFTs) are required per
iteration. In this paper, we propose an alternating
minimization algorithm for the continuous
version of the total  variation image deblurring problem.
We establish convergence of the proposed continuous alternating
minimization algorithm. The continuous setting is very useful to have a
unifying representation of the algorithm, independently of the discrete
approximation of the deconvolution problem, in particular concerning the strategies for dealing with boundary artifacts.
Indeed, an accurate restoration of blurred and noisy images requires a proper
treatment of the boundary.
A discrete version of our continuous alternating minimization algorithm
is obtained following two different strategies:
the imposition of appropriate boundary conditions and the enlargement of the domain.
The first one is computationally useful in the case of a symmetric blur,
while the second one can be efficiently applied for a nonsymmetric blur.
Numerical tests show that our algorithm generates higher quality
images in comparable running times with respect to
the Fast Total Variation deconvolution algorithm.
\end{abstract}
{\bf AMS-SC}: 65F10, 65F15, 65Y20, 46.

\section{Introduction}\label{sec1}
The bond between mathematics and visual observations has deep roots, down to the very beginning of science and technology. Nowadays, image processing enters so many different areas of sciences such as engineering, biology, medical sciences, breaking through everyday life. The basic problem of image restoration, once that any kind of corruption has occurred, has been tackled with the aid of computer technology, whose development from one side relies on the implementation of old mathematical tools, such as classical Fourier Analysis, on the other side promotes new mathematical results and throws light on new theoretical as well as applied challenges.

Here we consider the Total Variation (TV) image deblurring problem by
minimizing the following energy functional
\begin{equation} \label {tv}
E(u):=\frac{\alpha}{2}\|\ch u-f\|_{L^2(\Omega)}^2 +  \int_\Omega |\nabla u|\, dx,
\end{equation}
where $\alpha>0$ is a
fidelity parameter,
$\Omega$ is an
open rectangular domain in $\mathbb{R}^2$, $\ch$  is a given
linear blurring operator, $f:\Omega\to\mathbb{R}$ is  the observed
image in $L^2(\Omega)$, $u$ is the unknown image to restore, and
$|\cdot|$ denotes the Euclidean norm \cite{V02,CS05}.
The second term in \eqref{tv} is the total variation of $u$ and represents the energy obstruction to high frequency noise affecting the original image which is out of reach to human eyes and thus made unfavorable. We merely mention that functional of this type arise in different topics such as Cheeger's sets in differential geometry \cite{BK03}, degenerate singular diffusion PDE and the $1$-Laplacian \cite{LCE07}, elastic plastic problems \cite{RT83}.

\noindent The blurring model is assumed to be space-invariant,
namely the Point Spread Function (PSF) is represented by a specific real bivariate
function $h(x-y)$, $x, y\in\Omega$, for some univariate function
$h(\cdot)$ \cite{HNO05}. According to the linear modelling proposed in
the literature \cite{G93}, the observed image $f$ and the original
image $u$ are described by the relation
\begin{equation} \label {nbf}
f(x) = \ch u(x) + \eta(x):= \int_\Omega h(x-s)u(s)ds + \eta(x),
\quad x\in\Omega,
\end{equation}
where the kernel $h$ is the PSF and $\eta$ is the noise; notice that $\inf_{u} E(u)>0$.

A first step in the classical approach is to discretize \eqref{nbf} as follows
\begin{equation} \label {sysf}
\v{f} = A\v{u} + \boldsymbol{\eta},
\end{equation}
with $A \in \R^{m\times n}$ and $\v{u},\v{f},\boldsymbol{\eta}\in\R^n$.
The structure of the matrix $A$ is crucial to define fast deconvolution algorithms since $n$ and $m$ are very large.
In the very last years a lot of interest has been devoted to the definition of algorithms that combine edge preserving strategies with an appropriate treatment of the boundary artifacts \cite{BDS11,S12,AF13,MRF13}.
In the literature
one finds
mainly three strategies in order to obtain both accurate and fast restorations:
\begin{enumerate}
\item[1)] 
Choose and then impose appropriate {\em boundary conditions (BCs)} so that $n=m$ and the matrix $A$ can be
usually diagonalized by fast trigonometric transforms: discrete Fourier transform in the case of periodic BCs \cite{HNO05}, discrete cosine transform for reflective BCs and when the blur is
symmetric in any direction (quadrantally symmetric) \cite{NCT99,HNO05}, a low rank correction of the sine transform for antireflective BCs
can be exploited \cite{S03,ADNS11};
\item[2)] {\em Enlarge the domain} and we use periodic BCs on the larger domain, such that the computations can be carried out by FFTs
and eventually the image is  projected back to the original domain \cite{R05,DS10b,S12};
\item[3)] 
Work with the {\em underdetermined linear system} so that $m<n$: in such a setting the matrix $A$ can be represented as $A=MB$, where $M\in\R^{m\times n}$ is a mask that selects only the valid rows of $B\in\R^{n\times n}$ which can be diagonalized by FFT \cite{BB05,VBDW05,AF13,MRF13}.
\end{enumerate}
\noindent As it is well known, the linear system \eqref{sysf} is
a discrete ill-posed problem and
deconvolution algorithms are usually defined directly on the discrete setting.
For instance, the FTVd algorithm proposed in \cite{WYYZ08} regularizes the linear system \eqref{sysf}
by using  the following discrete version of the TV model \eqref{tv}
\begin{equation}
\min_{\mathbf{u}\in\R^{n}}  \frac{\alpha}{2}\|A \v{u}-\v{f}\|_{\ell^2}^2 +  \sum_{i=1}^{n} \|D_i \v{u}\|,
\end{equation}
where $D_i \v{u} \in \R^{n}$ denotes the discrete gradient of $\v{u}$ at the pixel $i$.

Conversely, in this paper we propose a different approach: we first regularize the continuous problem \eqref{tv}
and thereafter we
settle a discrete version of the continuous algorithm.
Buying the line of \cite{WYYZ08} for the discrete problem, we
set up and prove the convergence of an alternating minimization algorithm for solving \eqref{tv}.
Then, we provide and study its discrete version in connection with strategies 1) and 2).
Let us point out that if in the case of periodic BCs,
the discrete version of our algorithm is equivalent to the one proposed in \cite{WYYZ08}, the two algorithms turn out to be different in the case of antireflective BCs, since our proposal leads to the reblurring approach introduced in \cite{DS05,DEMS06}.
When the PSF is not quandrantally symmetric, accurate BCs, like reflective or antireflective, do not lead to matrices diagonalized by fast trigonometric transforms and hence strategy 2) could be a valid alternative.
On the other hand, when the PSF is quadrantally symmetric, it was theoretically proved (and numerically verified) in \cite{DS10b} that the imposition of reflective BCs and the enlargement of the domain by reflection are equivalent.
Therefore, in such a case the strategy in 1) has to be preferred because it yields the same restoration by fast transforms of smaller size.

\noindent Our approach has two advantages: it leads to fast computations, when antireflective boundary conditions are imposed for quadrantally symmetric PSFs, and it allows
a formal derivation and convergence analysis also for the strategy 2) with an enlarged domain. By working directly with the continuous formulation, we  avoid negligible details depending on the specific kind of discretization.
Unfortunately, for  strategy 3), the arising undetermined linear system cannot be easily treated and appropriate preconditioners should be investigated or different approachs like those in \cite{AF13,MRF13} should be
developed.

The paper is organized as follows. In Section \ref{sec2} we
set up our alternating minimization algorithm and we provide a convergence analysis in the continuous setting. In Section \ref{sec3} we propose two different discretization techniques, the first one based on the imposition of BCs and the second one on the enlargement of the domain. Numerical examples are reported and critically discussed in Section \ref{sec4} and concluding remarks are given in Section \ref{sec5}.

\section{An alternating minimization algorithm}\label{sec2}
In this section we reformulate the minimization problem for \eqref{tv} as a convex constrained minimization which in turn is solved by an alternating minimization algorithm.
At the end of the section, we provide a convergence analysis of the proposed algorithm.

\subsection{Reformulation of the problem}
As in \cite{WYYZ08}, we consider the following constrained convex minimization
\begin{equation} \label {tv-2}
\min\limits_{u,\bz}  \frac{\alpha}{2} \|\ch u-f\|_{L^2(\Omega)}^2 +  \int_\Omega |\bz|\, dx,\qquad \bz(x)=\nabla u(x),\quad x\in\Omega.
\end{equation}
By using the well-known quadratic penalization techniques,
we obtain the following convex minimization problem
\begin{equation} \label {tv-3}
\min\limits_{u,\bz}\, g(u,\bz):= \min\limits_{u,\bz}\,\frac{\alpha}{2} \|\ch u-f\|_{L^2(\Omega)}^2 +  \int_\Omega |\bz|\, dx
+ \frac{\beta}{2} \int_\Omega |\bz-\nabla u|^2\, dx,
\end{equation}
where $\beta>0$ is the penalization parameter. The solution of
problem (\ref{tv-3}) converges to that of problem (\ref{tv-2}) as $\beta\to \infty$
(see for instance \cite{NW99}).

\noindent Let
\[
\cu:=\{u\in L^2(\Omega): \nabla u\in L^2(\Omega)\},
\]
be the classical Sobolev space $H^1(\Omega)$ (see \cite{AF}) and
\[
\cz:=\{\bz:=(z_1,z_2)\in L^2(\Omega)\times L^2(\Omega)\}.
\]
The functional $g(u,\bz)$ turns out to be well defined for $(u,\bz)\in \cu\times\cz$ and we look for
\[
g(\cu,\cz):=\min\limits_{u\in\cu, \bz\in\cz}g(u,\bz).
\]

\subsection{An alternating minimization algorithm}

In the spirit of Csisz\'{a}r and Tusn\'{a}dy \cite{CT84}, we solve
problem (\ref{tv-3}) by alternatively minimizing $g(u,\bz)$ with
respect to $\bz$ while fixing $u$, and vice versa. The alternating
minimization algorithm is described as follows:

\begin{algorithm}\label{alg}
Given an arbitrary $u^0\in \cu$. For
$k=0,1,\ldots$,
\begin{itemize}
\item[(i)] minimize $g(u^k,\bz)$ over $\bz\in \cz$ to get $\bz=\bz^{k}$, and then
\item[(ii)] minimize $g(u,\bz^k)$ over $u\in \cu$ to get $u=u^{k+1}$.
\end{itemize} 
\end{algorithm}

Clearly, for the sequences $\{u^k\}$ and   $\{\bz^k\}$ generated
by the alternating minimization algorithm, we have
\[
g(u^{k+1},\bz^{k+1})\le g(u^{k+1},\bz^k)\le g(u^k,\bz^k),\quad k=0,1,\ldots.
\]
In the  alternating minimization algorithm, we need to solve two
auxiliary problems. On the one hand, for a fixed $u\in\cu$, we
solve the  problem
\begin{equation} \label {subp-1}
\min\limits_{\bz\in\cz}\, \int_\Omega |\bz|\, dx + \frac{\beta}{2} \int_\Omega |\bz-\nabla u|^2\, dx.
\end{equation}
The first-order optimality condition for (\ref{subp-1}) is given
by \BE\label{subp1-op} \bz = {\bf 0}\quad\mbox{or}\quad
\frac{\bz}{|\bz|}+\beta(\bz-\nabla u)={\bf 0}. \EE It is easy to
check that the solution to (\ref{subp1-op}) takes the form
\BE\label{subp1-sol} \bz = \max\left\{|\nabla
u|-\frac{1}{\beta},0\right\}\frac{\nabla u}{|\nabla u|}, \EE where
we set $0\cdot (0/0)=0$.

\noindent On the other hand, for a fixed $\bz\in\cz$, we proceed
by solving the minimization problem
\begin{equation} \label {subp-2}
\min\limits_{u\in\cu}\, \frac{\alpha}{2} \|\ch u-f\|_{L^2(\Omega)}^2  + \frac{\beta}{2} \int_\Omega |\bz-\nabla u|^2\, dx.
\end{equation}
The first-order optimality condition for (\ref{subp-2})  satisfies the Euler-Lagrange equation:
\[
\alpha\ch^*(\ch u-f) - \beta \nabla \cdot (\nabla u -\bz)=0
\]
that is
\BE\label{subp2-op}
\ch^*\ch u- \frac{\beta}{\alpha} \Delta u =\ch^*f - \frac{\beta}{\alpha} \nabla\cdot\bz
\EE
subject to homogeneous Neumann boundary conditions (BCs), where $\ch^*$ denotes the adjoint of the operator $\ch$ and $\Delta=\nabla \cdot \nabla$ is the Laplace operator.

\noindent For the operator $\ch$ defined in (\ref{nbf}), the adjoint $\ch^*$ is defined as
\[
\langle p,\ch^*q\rangle_{L^2(\Omega)} = \langle \ch p,q\rangle_{L^2(\Omega)},
\]
for all smooth functions $p,q$ with compact support in $\Omega$, where $\langle \cdot,\cdot\rangle_{L^2(\Omega)}$ is the inner product on $L^2(\Omega)$ defined as
\[
\langle p,q\rangle_{L^2(\Omega)} = \int_\Omega p(x)\overline{q(x)}dx,  \ \ \ p,q \in L^2(\Omega).
\]
Notice that
\begin{eqnarray*}
\langle \ch p,q\rangle_{L^2(\Omega)} &=& \int_\Omega \left(\int_\Omega h(x-s)p(s)ds\right)\overline{q(x)}dx \\
   &=& \int_\Omega p(s)\left(\int_\Omega h(x-s)\overline{q(x)}dx \right)ds\\
   &=& \int_\Omega p(s)\overline{\left(\int_\Omega \overline{h(x-s)}\,q(x)dx \right)}ds
\end{eqnarray*}
and thus
\[
\ch^*q(s) = \int_\Omega \overline{h(x-s)}\,q(x)dx.
\]
Therefore, if the kernel $h$ is a real bivariate function $h(x-y)$, $x, y\in\Omega$, like it happens in image
deblurring problems, then we have
\BE\label{eq:Hstar}
\ch^*u(x) = \int_\Omega h(s-x)u(s)ds = \int_\Omega h(-(x-s))u(s)ds.
\EE

\noindent Similarly, the adjoint $\nabla^*$ of  the differential operator $\nabla$ satisfies
\[
\langle p,\nabla^*\cdot q\rangle_{L^2(\Omega)} = \langle \nabla\cdot p,q\rangle_{L^2(\Omega)}
\]
for all smooth functions $p,q$ with compact support in $\Omega$. Then we have
\BE\label{eq:Nstar}
\nabla^*\cdot q = - \nabla\cdot q.
\EE
In fact, it is easy to verify that
\[
q\nabla\cdot p-p\nabla^*\cdot q= \nabla\cdot w,
\]
where $w:=p\,q$.  By the well-known divergence theorem, we get
\[
\int_\Omega (q \nabla\cdot p-p\nabla^*\cdot q)dx = \int_\Omega \nabla\cdot w dx
= \int_{\partial\Omega} w\cdot\bn\, dS=0,
\]
where $\bn$ is the outward pointing normal on $\partial\Omega$. As a consequence, equation (\ref{subp2-op}) takes the following form
\BE\label{subp2-op2}
\ch^*\ch u + \frac{\beta}{\alpha} \nabla^* \cdot \nabla u =\ch^*f + \frac{\beta}{\alpha} \nabla^*\cdot\bz.
\EE
Summarizing, the alternating minimization Algorithm \ref{alg} can be explicitly written as follows: given an arbitrary $u^0\in \cu$, for $k=0,1,\ldots$
\begin{itemize}
\item[(i)]  we set
\begin{equation}\label{eq:alg1}
\bz^{k} = \max\left\{|\nabla u^k|-\frac{1}{\beta},0\right\}\frac{\nabla u^k}{|\nabla u^k|},
\end{equation}
\item[(ii)] we compute $u^{k+1}$ by solving
\begin{equation}\label{eq:alg2}
\ch^*\ch u^{k+1} + \frac{\beta}{\alpha} \nabla^* \cdot \nabla u^{k+1} = \ch^*f + \frac{\beta}{\alpha} \nabla^*\cdot\bz^k.
\end{equation}
\end{itemize}

\subsection{Convergence Analysis}
In this section, we shall establish the convergence of the
proposed alternating minimization Algorithm \ref{alg}.  Let
$\theta:\cu\times\cu\to \R^+$ defined by \BE\label{fun:h} \theta(u,u'):=\frac{\alpha}{2}\|\ch
(u-u')\|_{L^2(\Omega)}^2 + \frac{\beta}{2} \int_\Omega |\nabla
u-\nabla u'|^2\, dx, \quad \forall u,u'\in\cu. \EE Notice that
$\theta(u,u)=0$ for all $u\in \cu$.

\noindent In order to prove the convergence of the 
Algorithm \ref{alg}, we proceed by proving some preliminary lemmas. Let us
first recall the following result from \cite{CT84}:
\begin{lemma}\label{lem:abn}
For $k=0,1,\ldots$, let $a^k$ and $b^k$ be extended real numbers greater than $-\infty$ and $c$ is a finite real number such that
\[
a^k+b^k\le b^{k-1}+c,\quad k=1,2,\ldots
\]
and
\[
\limsup_{k\to\infty}b^k>-\infty,\quad b^{k_0} <+\infty, \quad \mbox{for some } k_0.
\]
Then
\[
\liminf_{k\to\infty}a^k\le c.
\]
\end{lemma}
The following two lemmas rely on the definition of 
the function $\theta$  expressed in (\ref{fun:h}) and are crucial in proving Lamma \ref{lem:five}.
\begin{lemma}\label{lem:three}
Let $\{u^k\}$ and $\{\bz^k\}$ be the sequences generated by
Algorithm \ref{alg}. Then \BE\label{eq:three}
\theta(u,u^{k+1})+g(u^{k+1},\bz^k)\le g(u,\bz^k) \EE for all
$u\in\cu$ and  $k=0,1,\ldots$ .
\end{lemma}
\begin{proof}
By assumption, $g(u^{k+1},\bz^k)\le g(u,\bz^k)<+\infty$ for all $u\in\cu$. Note that $u_t:=(1-t)u+tu^{k+1}\in \cu$ for all $0<t\le 1$. Then the minimum of the function
\[
\phi(t):=g(u_t,\bz^k)
\]
is achieved for $t=1$. Hence,
\begin{eqnarray}\label{phi:dq}
0\ge \frac{\phi(1)-\phi(t)}{1-t} &=&\int_\Omega \frac{1}{1-t}\Big[\big(\frac{\alpha}{2}(\ch u^{k+1}-f)\overline{(\ch u^{k+1}-f)} + \frac{\beta}{2}  |\bz^k-\nabla u^{k+1}|^2\big) \nonumber \\
&&\quad - \big(\frac{\alpha}{2} (\ch u_t-f)\overline{(\ch u_t-f)} + \frac{\beta}{2} |\bz^k-\nabla u_t|^2\big)\Big]\, dx,
\end{eqnarray}
where the integrand is a difference quotient of the convex function
\[
\frac{\alpha}{2} (\ch u_t-f)\overline{(\ch u_t-f)} + \frac{\beta}{2} |\bz^k-\nabla u_t|^2
\]
of $t$, which is non-increasing as $t\to 1$.
As $t\to 1$ in (\ref{phi:dq}), using Lebesgue's monotone convergence theorem, we get
\begin{eqnarray*}
0&\ge& \int_\Omega \frac{d}{dt}\big(\frac{\alpha}{2} (\ch u_t-f)\overline{(\ch u_t-f)} + \frac{\beta}{2} |\bz^k-\nabla u_t|^2\big)_{t=1} \, dx \\
&=& \frac{\alpha}{2} \int_\Omega\big( (\ch u^{k+1}-f)\overline{(\ch u^{k+1}-f)}-(\ch u-f)\overline{(\ch u-f)}
+(\ch (u-u^{k+1})\overline{\ch (u-u^{k+1})}\big)\, dx\\
&&\quad + \frac{\beta}{2} \int_\Omega\big(|\bz^k-\nabla u^{k+1}|^2-|\bz^k-\nabla u|^2+ \nabla(u-u^{k+1})\overline{\nabla(u-u^{k+1})}\big) \, dx,
\end{eqnarray*}
which implies (\ref{eq:three}).
\end{proof}

\begin{lemma}\label{lem:four}
Let $\{u^k\}$ and $\{\bz^k\}$ be the sequences  generated by
Algorithm \ref{alg}. Then \BE\label{eq:four}
g(u,\bz^{k+1})\le \theta(u,u^{k+1})+g(u,\bz) \EE for all
$u\in\cu$, $\bz\in\cz$, and  $k=0,1,\ldots$ .
\end{lemma}
\begin{proof}
By assumption, $g(u^{k+1},\bz^{k+1})\le g(u^{k+1},\bz)<+\infty$ for all $\bz\in\cz$. Note that $\bz_t:=(1-t)\bz+t\bz^{k+1}\in\cz$ for all $0<t\le 1$. Then the minimum of the function
\[
\psi(t):=g(u^{k+1},\bz_t)
\]
is attained at $t=1$. Hence,
\BE\label{psi:dq}
0\ge \frac{\psi(1)-\psi(t)}{1-t} =\int_\Omega \frac{1}{1-t}\Big[\big(|z^{k+1}| + \frac{\beta}{2}  |\bz^{k+1}-\nabla u^{k+1}|^2\big) - \big(|\bz_t| + \frac{\beta}{2} |\bz_t-\nabla u^{k+1}|^2\big)\Big]\, dx,
\EE
where the integrand is a difference quotient of the convex function
\[
|\bz_t| + \frac{\beta}{2} |\bz_t-\nabla u^{k+1}|^2
\]
of $t$, which is non-increasing as $t\to 1$. Moreover, as c$t\to 1$ in (\ref{psi:dq}), using Lebesgue's monotone convergence theorem, we find
\begin{eqnarray}\label{ineq:four}
0&\le& -\int_\Omega \frac{d}{dt}\big(|\bz_t| + \frac{\beta}{2} |\bz_t-\nabla u^{k+1}|^2\big)_{t=1} \, dx \nonumber\\
&=& \int_\Omega\Big( \frac{\langle \bz^{k+1},\bz \rangle}{|\bz^{k+1}|}-|\bz^{k+1}|\Big)\, dx
+ \frac{\beta}{2} \int_\Omega 2\langle \bz^{k+1}-\nabla u^{k+1},\bz-\bz^{k+1}\rangle \, dx \nonumber\\
&\le& \int_\Omega\Big(|\bz|-|\bz^{k+1}|\Big)\, dx
+ \frac{\beta}{2} \int_\Omega 2\langle \bz^{k+1}-\nabla u^{k+1},\bz-\bz^{k+1}\rangle \, dx,
\end{eqnarray}
where $\langle\cdot,\cdot\rangle$ denotes the Euclidean inner product.

On the other hand, it is easy to check that
\begin{eqnarray*}
&&\frac{\beta}{2} \int_\Omega \Big(2\langle \bz^{k+1}-\nabla u^{k+1},\bz-\bz^{k+1}\rangle + |\bz^{k+1}-\nabla u|^2 -|\bz-\nabla u|^2-|\nabla u-\nabla u^{k+1}|^2\Big)\, dx\\
&=& -\frac{\beta}{2} \int_\Omega |(\bz^{k+1}-\bz)-(\nabla u^{k+1}-\nabla u)|^2\le 0.
\end{eqnarray*}

\noindent This, together with (\ref{ineq:four}), yields  (\ref{eq:four}).
\end{proof}

\noindent By combing Lemma \ref{lem:three} and Lemma \ref{lem:four}, we deduce the following result.
\begin{lemma}\label{lem:five}
Let $\{u^k\}$ and $\{\bz^k\}$ be the sequences generated by
Algorithm \ref{alg}. Then \BE\label{eq:five}
g(u,\bz^{k})+g(u^{k},\bz^{k})\le g(u,\bz) + g(u,\bz^{k-1}) \EE for
all $u\in\cu$, $\bz\in\cz$, and  $k=1,2,\ldots$ .
\end{lemma}
\begin{proof}
By adding  (\ref{eq:three}) and (\ref{eq:four}), we get
\[
g(u,\bz^{k})+g(u^{k},\bz^{k-1})\le g(u,\bz) + g(u,\bz^{k-1}),\quad k=1,2,\ldots.
\]
This, together with $g(u^{k},\bz^{k-1})\ge g(u^{k},\bz^{k})$, leads to (\ref{eq:five}).
\end{proof}

For any given $u\in\cu$, we have the following result on the monotonicity of $\theta(u,u^k)$ in $k$, where $\theta$ is defined in (\ref{fun:h}).
\begin{lemma}\label{lem:h-mono}
Let the sequences $\{u^k\}$ and $\{\bz^k\}$ be generated by
Algorithm \ref{alg}. Then, for any $u\in\cu$,
\BE\label{eq:h-mono}
\theta(u,u^{k+1})\le \theta(u,u^k)
\EE
for all  $k=0,1,\ldots$ .
\end{lemma}
\begin{proof}
By (\ref{eq:three}) and (\ref{eq:four}), we have
\[
\theta(u,u^{k+1})+g(u^{k+1},\bz^k)\le g(u,\bz^k)\le\theta(u,u^{k})+g(u,\bz)
\]
for all $u\in\cu$, $\bz\in\cz$, and  $k=0,1,\ldots$  .
If $u\in\cu$ and $\bz\in\cz$ are such that $g(u,\bz)=g(\cu,\cz)$, then $g(u^{k+1},\bz^k)\ge g(u,\bz)$. Thus (\ref{eq:h-mono}) follows.
\end{proof}


We are now in a position to prove that the alternating argument of
\cite[Theorems 1--3]{CT84} adapted to our situation yields an
$\varepsilon$-optimal solution  (actually a solution) to the
minimization problem, in the sense that
\begin{equation}\label{e_sol_def}
g(u^k,z^k)=g(\mathcal U,\mathcal Z)+\varepsilon
\end{equation}
where $\varepsilon>0$ can be made arbitrary small, provided $k$ is large enough. We have the following
\begin{theorem}\label{thm:am-conv}
The functional $g(\cdot,\cdot)$ has a global minimum in the space
$H^1(\Omega)\times L^2(\Omega)$ attained at a unique point
$(u,z)$.  Let $\{u^k\}$ and $\{\bz^k\}$ be the sequences generated
by Algorithm \ref{alg}. Then
\BE\label{eq:am-conv} \lim_{k\to\infty}g(u^k,\bz^{k}) = g(\cu,\cz)
\EE Moreover, the alternating sequence converges to the global
minimum of the functional $g$.
\end{theorem}
\begin{proof}
Observe that the functional $g$ is coercive, namely
$$g(u,\bz)\to \infty,\quad \text{ as }\quad  \|(u,\bz)\|:=\|u\|_{L^2}+\|\nabla u\|_{L^2}+\|\bz\|_{L^2}\to\infty$$
and weakly lower semi-continuous. Since $\cu\times\cz$ is a
reflexive Banach space, the global minimum of $g$ is attained and it is attained
at a unique point $(u,\bz)$ thanks to the fact that the functional
$g$ is strictly convex, see \cite{ET99}. Let $(u^k,\bz^k)$ be the alternating
sequence and let
\[
a^k=g(u^k,\bz^k),\quad b^k=g(u,\bz^k),\quad c=g(u,\bz)=g(\cu,\cz).
\]
By Lemma \ref{lem:five}, we have
\[
a^k+b^k\le b^{k-1}+c,\quad k=1,2,\ldots.
\]
By assumption, $0\le b^k<+\infty$ for all $k$. Thus, by using Lemma \ref{lem:abn}, we obtain
\BE\label{eq:g-liminf}
\liminf_{k\to\infty}g(u^k,\bz^k)\le g(\cu,\cz).
\EE
Moreover, we also have
\BE\label{eq:g-mono}
g(\cu,\cz) \le g(u^{k},\bz^{k})\le g(u^{k-1},\bz^{k-1}),\quad k=1,2,\ldots.
\EE
That is, $\{g(u^{k},\bz^{k})\}$ is non-increasing. By combining (\ref{eq:g-liminf}) and (\ref{eq:g-mono}),  we plainly deduce the
limit relation in  (\ref{eq:am-conv}).

Since $H^1(\Omega)\times L^2(\Omega)$ is reflexive, there exists a
subsequence $(u^{k_j},\bz^{k_j})$ which converges to the global
minimum.  Actually, we now show that the whole sequence
$(u^k,\bz^k)$ does converge. Indeed, by Lemma \ref{lem:h-mono},
$\{\theta(u,u^k)\}$ is monotone non-increasing with respect to
$k$, where $\theta$ is the metric defined in (\ref{fun:h}). Hence,
$\{\theta(u,u^k)\}$ converges as $k\to\infty$ and necessarily to
zero as one has  $\{u^{k_j}\}\to u$. Therefore, we have $u^k\to
u^*$ as well as as $\bz^k\to \bz$, as $k\to\infty$.
\end{proof}

\section{Discrete alternating minimization algorithms}\label{sec3}
In this section we provide two discrete versions of Algorithm \ref{alg}:
the first one is based on the imposition of proper BCs
whereas the second one relies on the enlargement of the domain technique.

In order to approximate equations (\ref{eq:alg1}) and  (\ref{eq:alg2}),
for the sake of simplicity, we assume that the
domain $\Omega$ is square  (the case of a rectangular domain is the same).
Let $\Omega_n$ be a $n \times n$ uniform grid on $\Omega$,
let $\v{u_n}\in\R^{n^2}$ denote the stack ordered unknown
image to be restored, which is  the collocation of the function $u$ at $\Omega_n$,
and let $\v{f_n}\in\R^{n^2}$ represent an observed grayscale image
which is  the stack ordered collocation of the function $f$ at $\Omega_n$.

\subsection{Alternating minimization algorithm by BCs}\label{sub:bc}
We impose the same BC on $u$ both on
integral and differential operators obtaining $n^2 \times n^2$
discrete operators.
Classical choices are:
\begin{itemize}
	\item {\em zero-Dirichlet}: $u(x)=0$, for $x\in \R^2\setminus\Omega$;
	\item {\em periodic}: $u$ periodically
extended outside $\Omega$;
	\item {\em reflective}: discretize $\partial u(x)=0$, for $x \in \partial \Omega$ (Neumann BCs), by symmetry
with respect to the midpoint;
	\item {\em antireflective}: discretize $\partial u(x)^{-}=\partial u(x)^{+}$, for $x \in \partial \Omega$,
with respect to the midpoint.
\end{itemize}
Reflective and antireflective BCs have been introduced in \cite{NCT99,S03}, respectively, in a more general setting
where $u(x)$ is not necessarily differentiable (neither continuous). Further details on their implementation can be found in \cite{HNO05,DS10}.

First, we discretize $\nabla$ to approximate equation \eqref{eq:alg1}.
Let $D_1, \, D_2\in\Rntnt$ be the two
first-order forward finite difference operators with appropriate
BCs in the $x$ and $y$ directions, respectivelly.  We use $\v{z_1},\v{z_2}\in\Rnt$ as the approximations to
$D_1\v{u_n}$ and $D_2\v{u_n}$, respectively. Define
$\bz=(\v{z_1}; \, \v{z_2})\in\Rtnt$ and
$D=(D_1; \, D_2)\in\R^{2n^2\times n^2}$. Also, let
$\bz_i=((\v{z_1})_i; \, (\v{z_2})_i)\in\R^2$ and
$D_i \v{u_n}=((D_1\v{u_n})_i; \, (D_2\v{u_n})_i)\in\R^2$, for
$i=1,\ldots,n^2$. Therefore, the discrete
approximation of  equation (\ref{eq:alg1}) is given \nolinebreak by
\BE\label{subp1-sol-d} \bz^k_i = \max\left\{|D_i
\v{u_n}^k|-\frac{1}{\beta}, \, 0\right\}\frac{D_i \v{u_n}^k}{|D_i \v{u_n}^k|},\qquad
i=1,\ldots,n^2, \EE where $0\cdot(0/0)$ is set to be $0$.

Now, we discretize the equation  \eqref{eq:alg1}.
The integral equation $\ch u$ is approximated by a collocation method: we obtain
$H\in\Rntnt$ as the discretization matrix of
the blurring (or convolution) operator $\ch$ with appropriate BCs.
Similarly, thanks to equation \eqref{eq:Hstar}, we define $H'\in\Rntnt$ as the discretization matrix of
the correlation operator $\ch^*$ with the same BCs.
Note that $H'=H^T$ for zero-Dirichlet and periodic BCs, while $H'\neq H^T$ for antireflective BCs \cite{DEMS06}. In the case of reflective BCs $H'=H^T$ if the PSF is quadrantally symmetric and $H'\neq H^T$ otherwise.
Similarly, thanks to \eqref{eq:Nstar}, the discretization of $\nabla^*$ leads to
$D_1', \, D_2'\in\Rntnt$ and
$D'=(D_1', \, D_2')\in\R^{n^2\times 2n^2}$.
Therefore,  the discrete
approximation of  equation (\ref{eq:alg2}) is given by
\BE\label{subp2-op-d}
\left(H'H - \frac{\beta}{\alpha} D'D\right) \v{u_n}^{k+1} = H'\v{f_n} - \frac{\beta}{\alpha} D'\bz^k.
\EE

The idea of using the operator $H'$ instead of $H^T$ was introduced in \cite{DEMS06}, where it was called {\em reblurring}.
This strategy turns out to be
computationally very useful in the case of antireflective BCs and quadrantally symmetric PSFs. Indeed,
in such a case
$H'H$ can be diagonalized by the antireflective transform \cite{ADNS11},
$H^TH$ has  a structure not easily invertible \cite{ADS08}. Therefore, the discrete alternating minimization algorithm defined by  \eqref{subp1-sol-d} and \eqref{subp2-op-d} coincide with the one proposed in \cite{WYYZ08}  in the case of periodic BCs, however the two algorithms are different in the case of antireflective BCs.

\subsection{Enlargement of the domain}
Exploiting suitable BCs
allows to derive square discrete operators, which are particularly useful to obtain a fast solution of the linear system \eqref{subp2-op-d},
defining $u(x)$ for $x\in\R^2\setminus\Omega$ as a linear combination of values of $u$ inside $\Omega$. In this subsection we use a different strategy for computing the undefined values of $u(x)$ for $x\in\R^2\setminus\Omega$. We expand the domain $\Omega$ to a new domain $\widetilde\Omega$ such that $\Omega \subset \widetilde\Omega$ and where $\widetilde\Omega$ is large enough to contain all points necessary to define $u(x)$ for $x\in\Omega$. Therefore, the size of $\widetilde\Omega$ depends on the size of the support of the kernel function $h(x)$ in \eqref{nbf}, i.e., the PSF,
and in the worst case it is twice the size of $\Omega$, since $h(x)$ has compact support.

\noindent First, we define a projection operator $\mathcal{P}_{\Omega}^{\widetilde\Omega}: \Omega \to \widetilde\Omega$, such that $\mathcal{P}_{\Omega}^{\widetilde\Omega}f(x)=f(x)$ for all $x \in \Omega$. For $x\in\widetilde\Omega\setminus\Omega$ the value of $\mathcal{P}_{\Omega}^{\widetilde\Omega}f(x)$ can be defined following the same choices used for the BCs. It can be fixed equal to zero or obtained as a periodic
extension (reflection or antireflection)
of $f(x)$ with $x\in\Omega$. Of course, other
extrapolation strategies, like the one proposed
for the synthetic BCs \cite {FN11}, could be used
provided the computational cost does not exceed the one of the FFT.

\noindent Once that $\widetilde\Omega$ and $\mathcal{P}_{\Omega}^{\widetilde\Omega}$ are defined, the minimization problem \eqref{tv} can be reformulated on  $\widetilde\Omega$ instead of $\Omega$ and than discretized with appropriate BCs like in Subsection~\ref{sub:bc}. Eventually, the restored image is the inner part of the approximated solution at the grid points
inside $\Omega$.
Since the
use of accurate BCs is useful to remove ringing effects (Gibbs phenomena),
which reduce moving towards the interior of $\widetilde\Omega$,
and the boundary of $\Omega$ is far away enough from the boundary of $\widetilde\Omega$ (depending on the support of the PSF), accurate BCs are not longer necessary on $\widetilde\Omega$ and a computationally cheap choice can be adopted: in fact, precise BCs could be important only if the discrete Toeplitz operators
associated
to the PSF are heavily ill-conditioned in the high
frequency
domain which in turn is
is due to the presence of zeros
of high order close to the boundaries of $(-\pi,\pi]^2$ of the symbol induced by the PSF. Hence, periodic BCs are usually applied on $\widetilde\Omega$ such that all computations can be performed by FFT also when the PSF is not symmetric. It follows that the matrices involved in equations (\ref{subp2-op-d})  are
diagonalizable
by FFT and $H'=H^T$ as observed at the end of Subsection \ref{sub:bc}.

\section{Numerical results}\label{sec4}
In this section, we report the performances of the considered discrete versions of our alternating
minimization Algorithm \ref{alg} described in Section \ref{sec3}.
According to the analysis in \cite{DS10b}, the two discrete algorithms provide the same restoration when the PFS is quadrantally
symmetric, but they have a different computational cost. Indeed, the second strategy based on the enlargement of the domain is more expensive, since the discretization of $\widetilde\Omega$ leads to a larger linear
system
and thus the FFTs (or the discrete cosine transforms) are applied to larger vectors. On the other hand, when the PSF is not quadrantally symmetric, the linear system in \eqref{subp2-op-d} can not be diagonalized by fast transforms in the case of accurate
BCs
(reflective or antireflective) and it should be solved by preconditioned conjugate gradient with proper preconditioners \cite{CCW99,BDS11}. The high order BCs proposed in \cite{D10} could solve the linear system \eqref{subp2-op-d} by FFTs plus some lower order computations, but they are purely algebraic and can not be interpreted as the discretization of a continuous problem, and hence they do not fit with the framework proposed in this paper. Differently, the enlargement of the domain does not depend on the symmetry of the PSF and the larger linear systems can be solved again by FFTs. Therefore, we suggest to use the BCs approach, when the PSF is quadrantally symmetric and by the enlargement of the domain, when the PSF is not quadrantally symmetric. We provide two examples of these different situations. Moreover, we compare the different kind of BCs or
extensions,
namely periodic, reflective and antireflective, obtaining, according to results in the literature (cf. \cite{DS10}), that the antireflective strategy provides better restorations and is more stable, when varying the
fidelity parameter $\alpha$.

Accordingly
to the reformulation \eqref{tv-3} of the TV minimization problem \eqref{tv}, the parameter $\beta$ has to be chosen large enough. On the other hand, when $\beta$ is large our alternating minimization algorithm converges slowly. Thanks to the huge numerical experimentation in \cite{WYYZ08}, whose proposal corresponds to our algorithm in the case of periodic BCs, we can safely choose $\beta=2^7$ and we can implement a continuation strategy on $\beta$. This means that we add an outer iteration on $\beta = 2^j$, for $j=1,\dots,7$, using as initial guess $\v{u}_\v{n}^0$ the restoration computed by the previous value of $j$, for $j=1$ the initial guess is the observed image $\v{f}$. In this way, we have a much faster convergence, because for small $j$ the alternating minimization algorithm converges quickly even if the restoration is not enough accurate, while for large $j$ we obtain an accurate restoration speeding up the convergence thanks to the good initial guess.

Our test examples are constructed from the true image and the PSF, computing a blurred image by convolution. Hence we cut the inner part of the blurred image such that the pixels of the new smaller observed image are not affected by the BCs of the convolution, but only by the other pixels of the true image. Of course, the number of pixels cut at the boundary depends on the support of the PSF. The
tailoring is marked in the true image by a white rectangle in Figures \ref{fig:im-tb} and \ref{fig:im-tb2}. Note also the different size of the true (larger) image and the blurred image to restore. Then, we add the Gaussian noise to the blurred image, where the Gaussian noise has zero mean and different variances: $\sigma^2=10^{-6}$, $\sigma^2=10^{-4}$,  or $\sigma^2=5\times 10^{-4}$.  As in \cite{WYYZ08}, the quality of restoration is measured by the signal-to-noise ratio (SNR):
\[
{\rm SNR}:=10\ast\log_{10}\frac{\|\v{u_n}-\bar{\v{u}}_\v{n}\|^2}{\|\v{u_n}-\hat{\v{u}}_\v{n}\|^2},
\]
where $\v{u_n}$ is the original image, $\bar{\v{u}}_\v{n}$ is the image with the mean value of $\v{u_n}$ as pixels, and $\hat{\v{u}}_\v{n}$ is the restored image.

The choice of the
fidelity parameter $\alpha$ is a difficult task and it should be further investigated. Here, we solve the minimization problem for several values of $\alpha$, selecting by hand the parameter that provides the maximum SNR. We consider also the choice proposed in \cite{WYYZ08} that is $\alpha = 0.05/\sigma^2$, but the following numerical results show that this is not always a good choice since $\alpha$ depends both on the noise, i.e., $\sigma^2$, and on the PSF.

All the tests are carried out by using {\tt MATLAB 7.12} on a personal computer Intel(R) Core(TM)2 Duo of 1.80 GHz CPU.

\begin{figure}[tb]
\begin{center}
    \begin{center}
    \begin{minipage}[c]{5.0cm}
        \centering
        \epsfig{figure=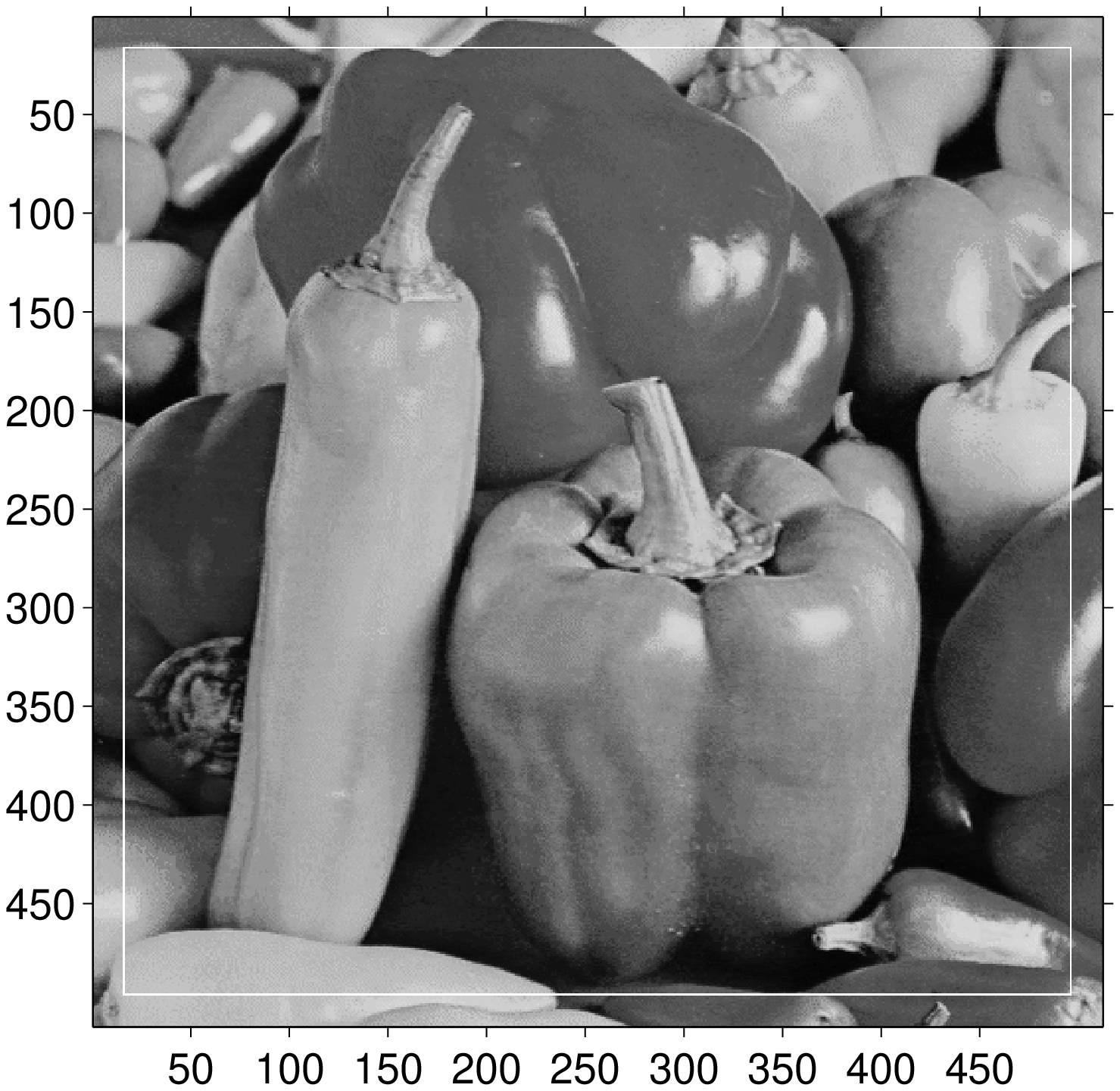,width=5.0cm}
        \small{True image}
    \end{minipage}
    \begin{minipage}[c]{5.0cm}
        \centering
        \epsfig{figure=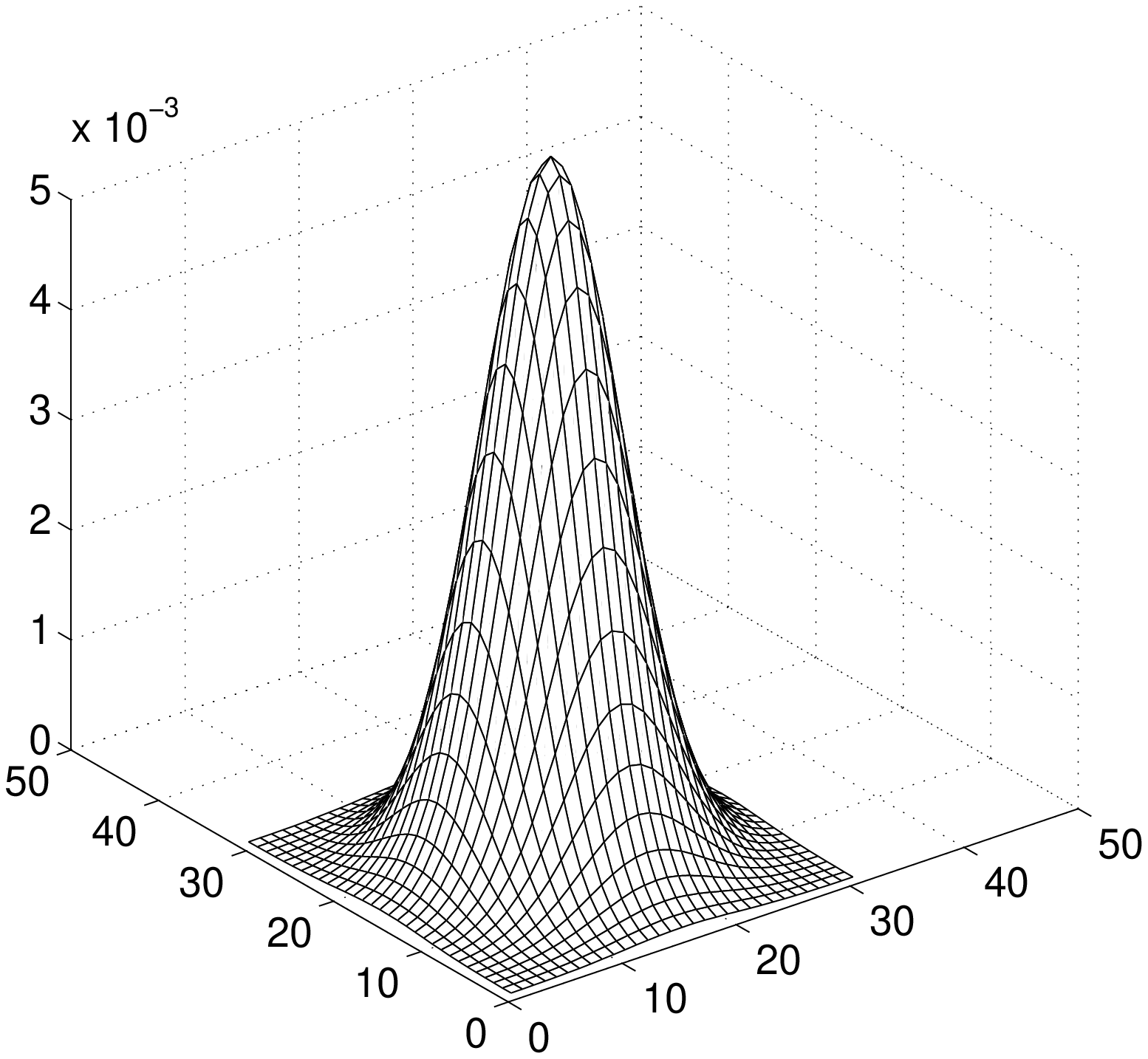,width=5.0cm}
        \small{PSF with ${\tt hsize}=16$ and $\delta=5$}
   \end{minipage}
    \begin{minipage}[c]{5.0cm}
        \centering
        \epsfig{figure=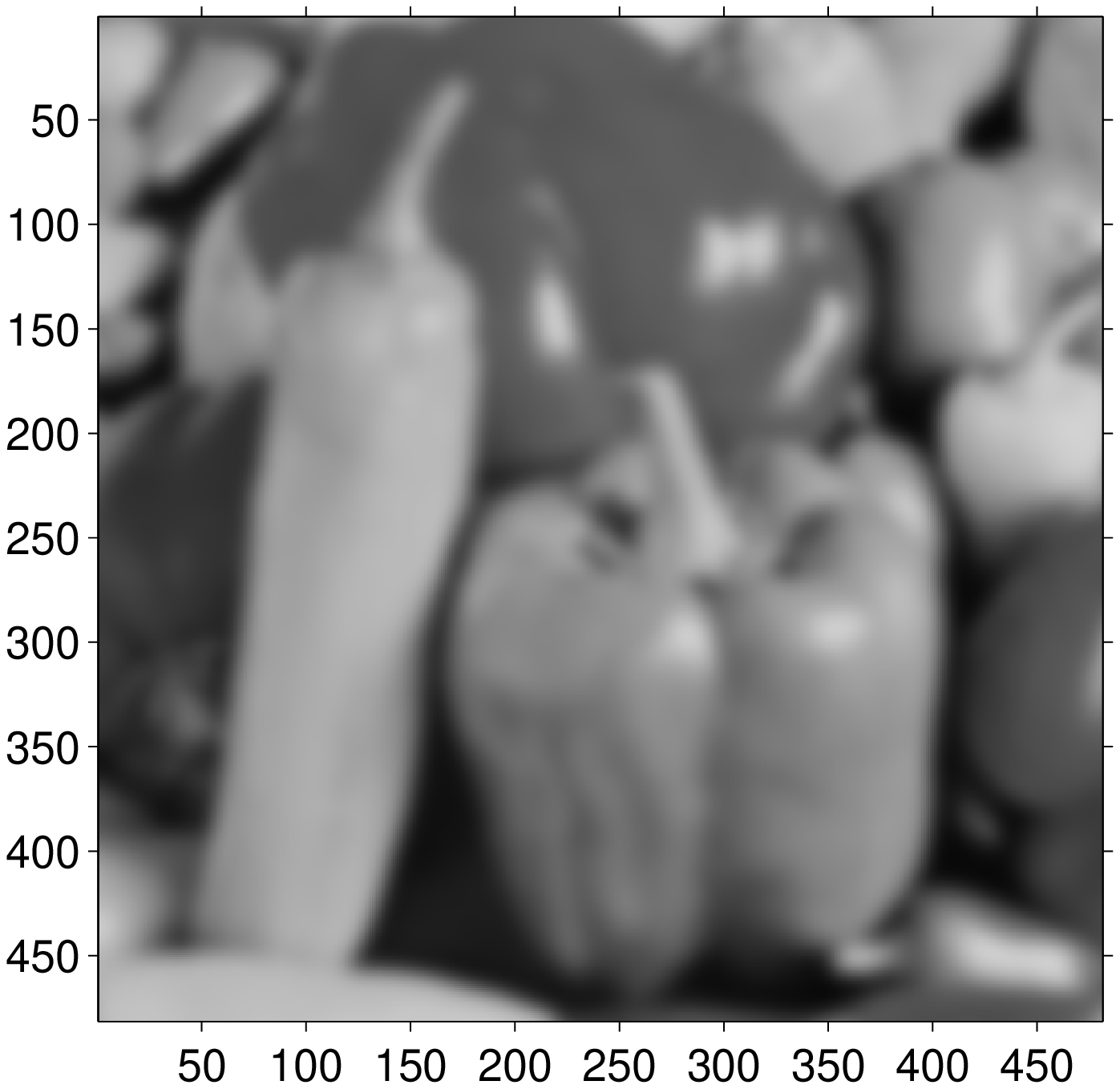,width=5.0cm}
        \small{Blurred image}
    \end{minipage}\\
    \begin{minipage}[c]{5.0cm}
        \centering
        \epsfig{figure=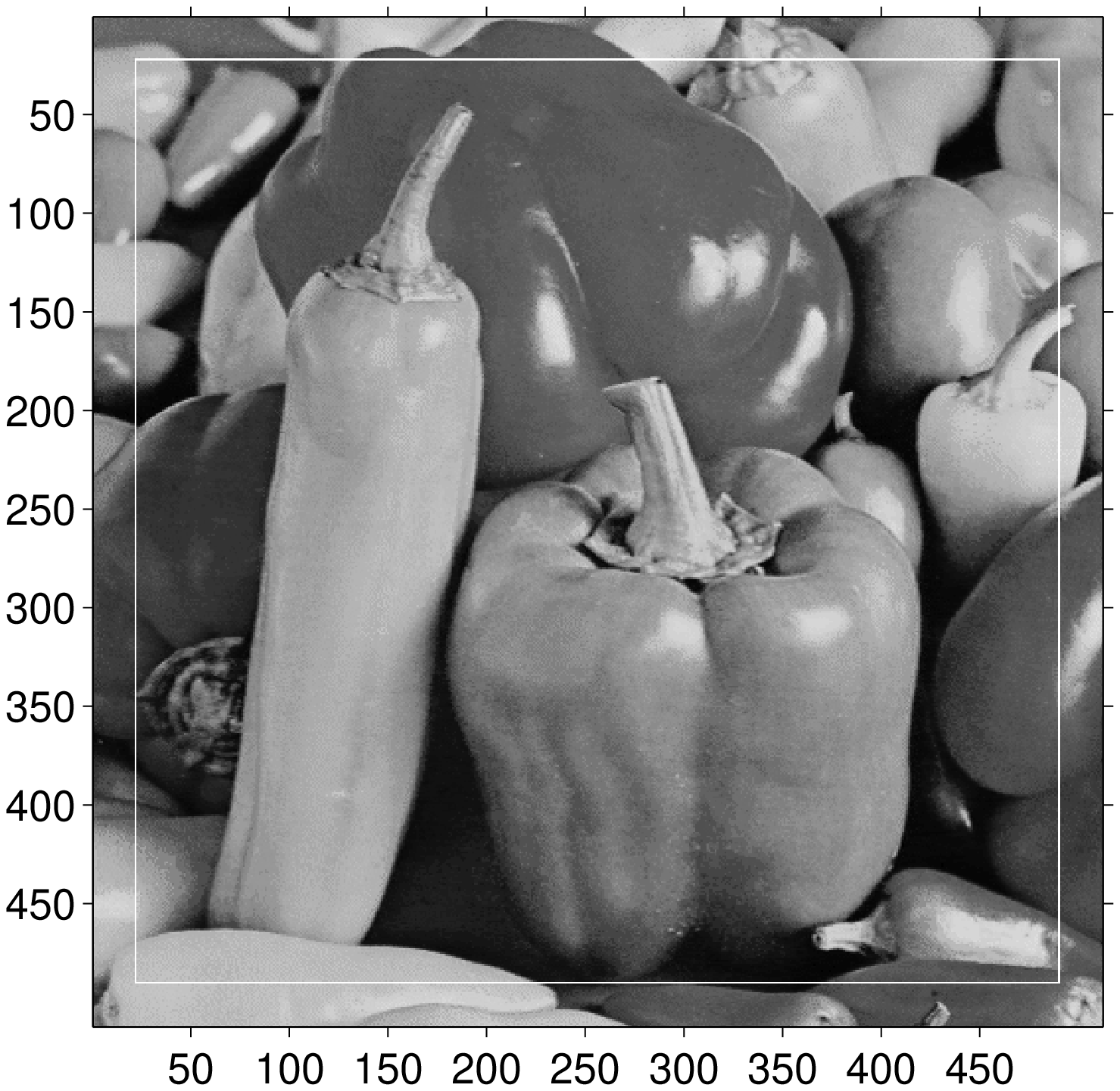,width=5.0cm}
        \small{True image}
    \end{minipage}
    \begin{minipage}[c]{5.0cm}
        \centering
        \epsfig{figure=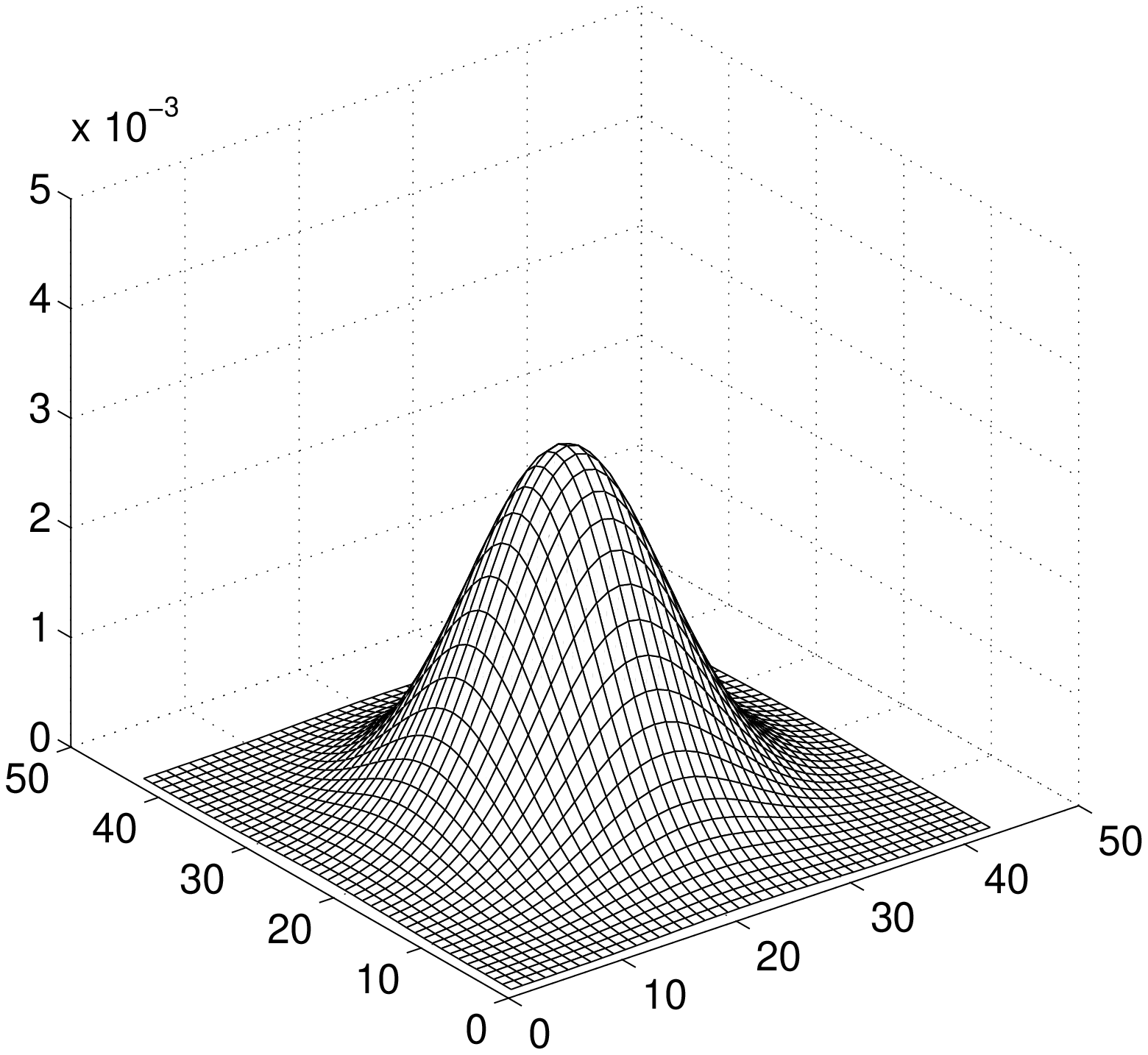,width=5.0cm}
        \small{PSF with ${\tt hsize}=22$ and $\delta=7$}
   \end{minipage}
    \begin{minipage}[c]{5.0cm}
        \centering
        \epsfig{figure=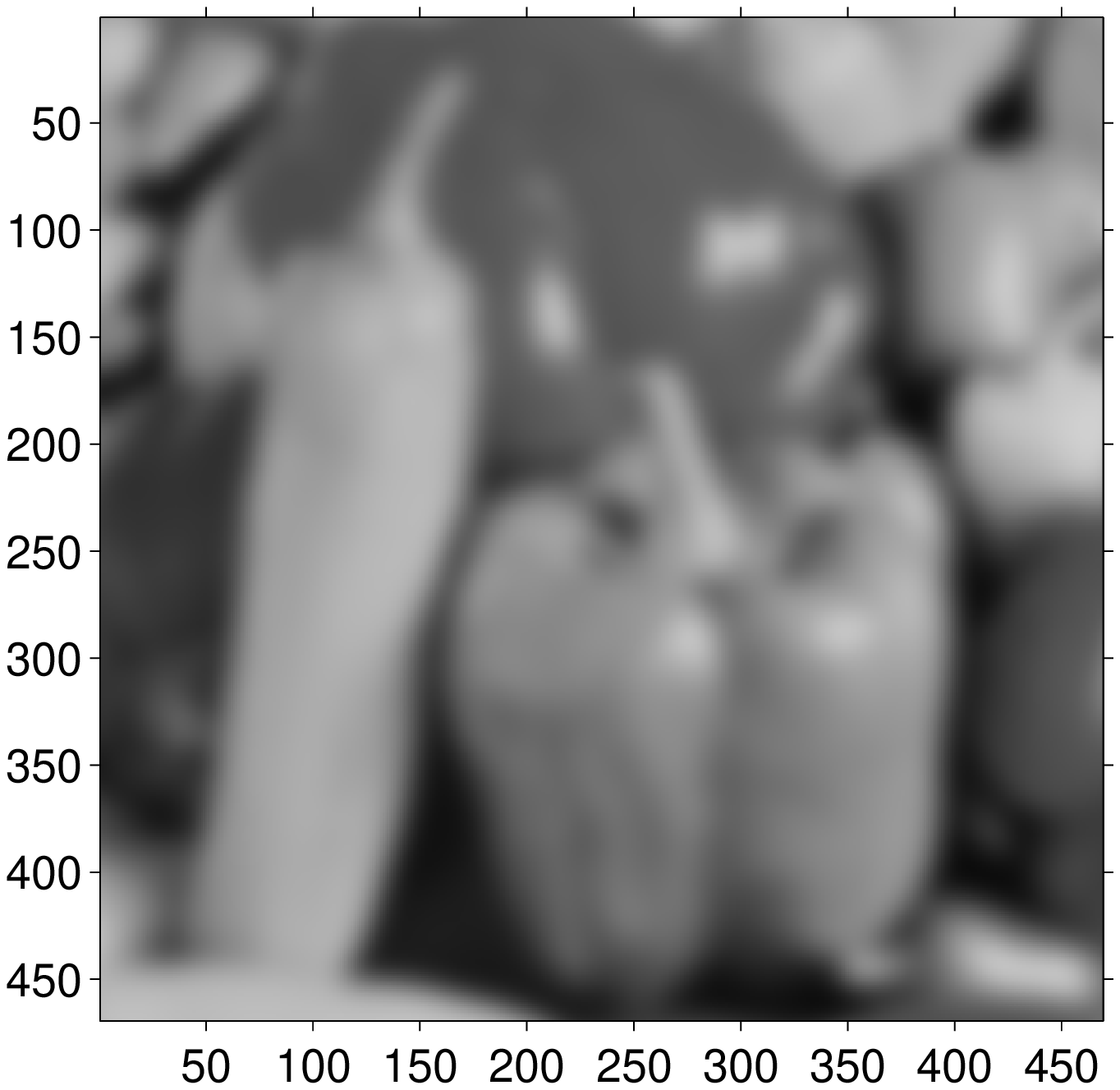,width=5.0cm}
        \small{Blurred image}
    \end{minipage}
    \end{center}
\end{center}
\caption{True images, PSF, and blurred images without noise.} \label{fig:im-tb}
\end{figure}

\subsection{Quadrantally symmetric PSF}
In this example we consider a symmetric Gaussian PSF with different blurring sizes and standard deviations: a) ${\tt hsize}=16$ and $\delta=5$ and b) ${\tt hsize}=22$ and $\delta=7$.
Figure \ref{fig:im-tb} shows the true image, the PSFs with different blurring sizes and standard deviations, and the blurred images without noise. The white box denotes the field of view.

Figure \ref{fig:snr-c} shows the SNR versus the regularization parameter $\alpha$ for different BCs and different noise levels. We note that, as expected, antireflective BCs provides the higher SNR for all the considered values of $\alpha$, in particular around the value of $\alpha$ that gives the maximum SNR. Since the restoration error has a component related to the boundary artifacts and
another component associated to the noise, the use of accurate BCs, like antireflective, is especially useful for a low level of the noise. Figure \ref{fig:snr-c} shows also that the value of $\alpha=0.05/\sigma^2$ proposed in \cite{WYYZ08} is a good choice only for the antireflective BCs with low noise ($\sigma^2=10^{-6}$).

\begin{figure}[tb]
\begin{center}
    \begin{center}
    \begin{minipage}[c]{5.0cm}
        \centering
        \epsfig{figure=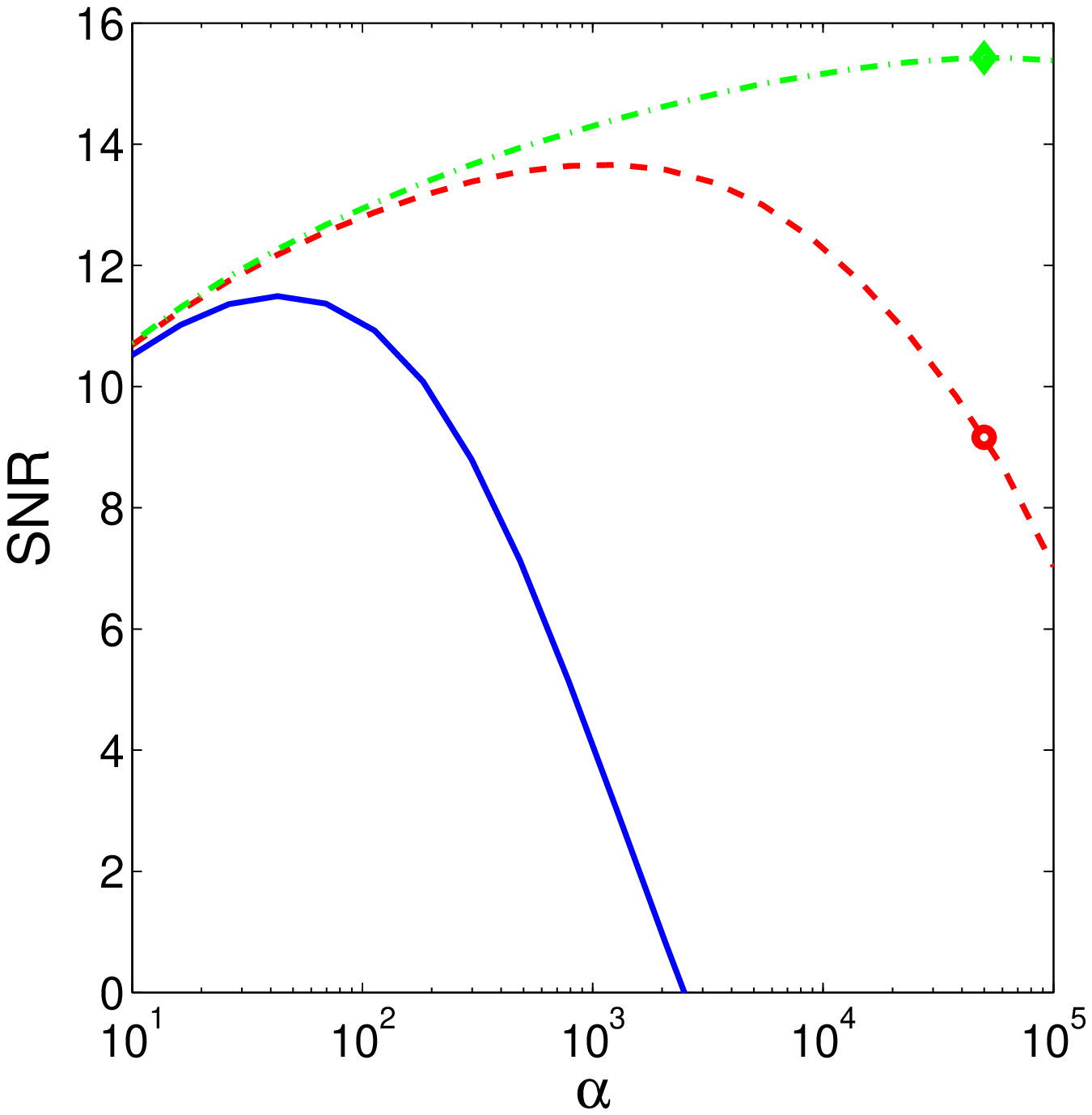,width=5.0cm}
        \small{$\sigma^2=10^{-6}$}
    \end{minipage}
    \begin{minipage}[c]{5.0cm}
        \centering
        \epsfig{figure=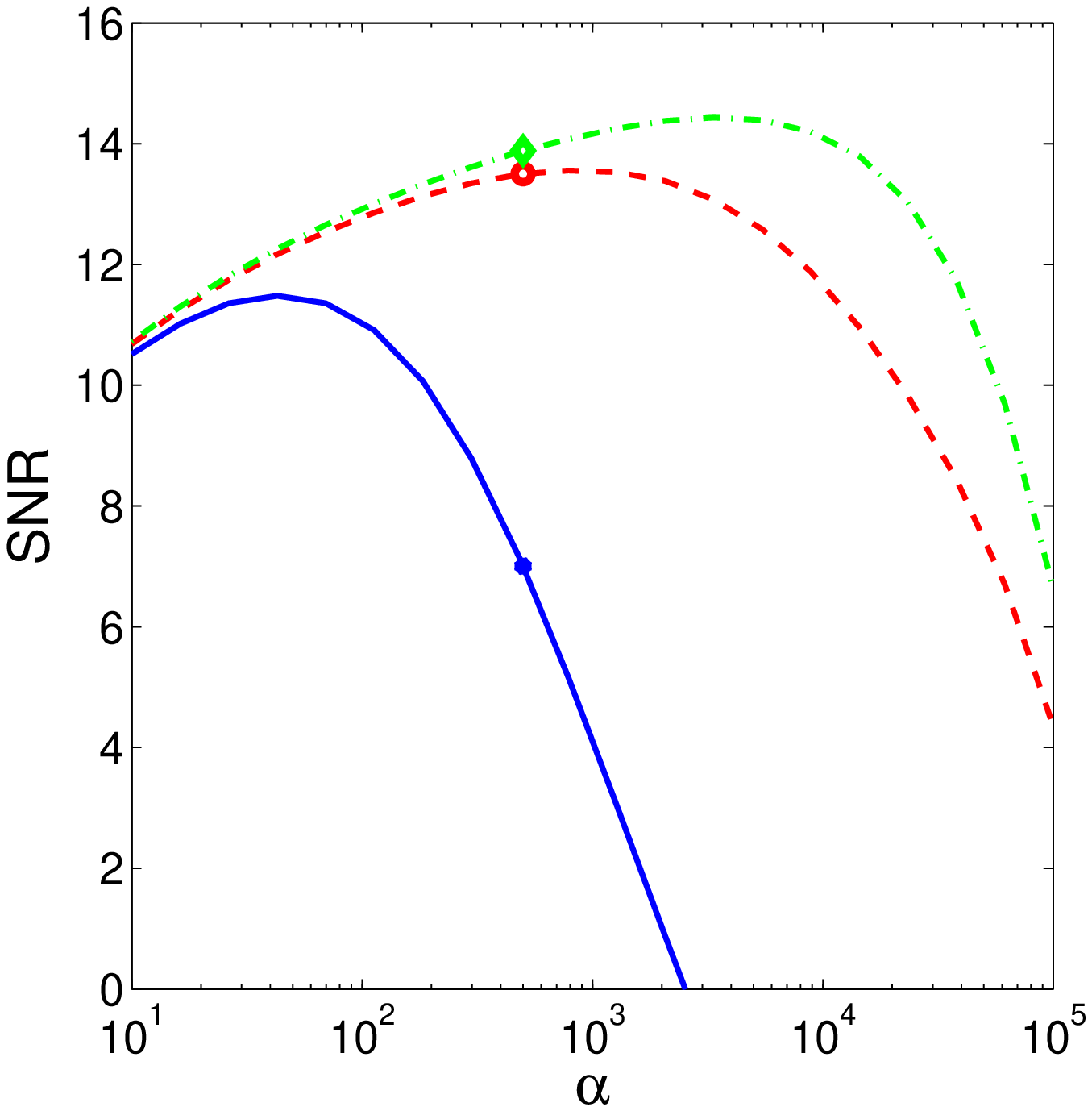,width=5.0cm}
        \small{$\sigma^2=10^{-4}$}
   \end{minipage}
    \begin{minipage}[c]{5.0cm}
        \centering
        \epsfig{figure=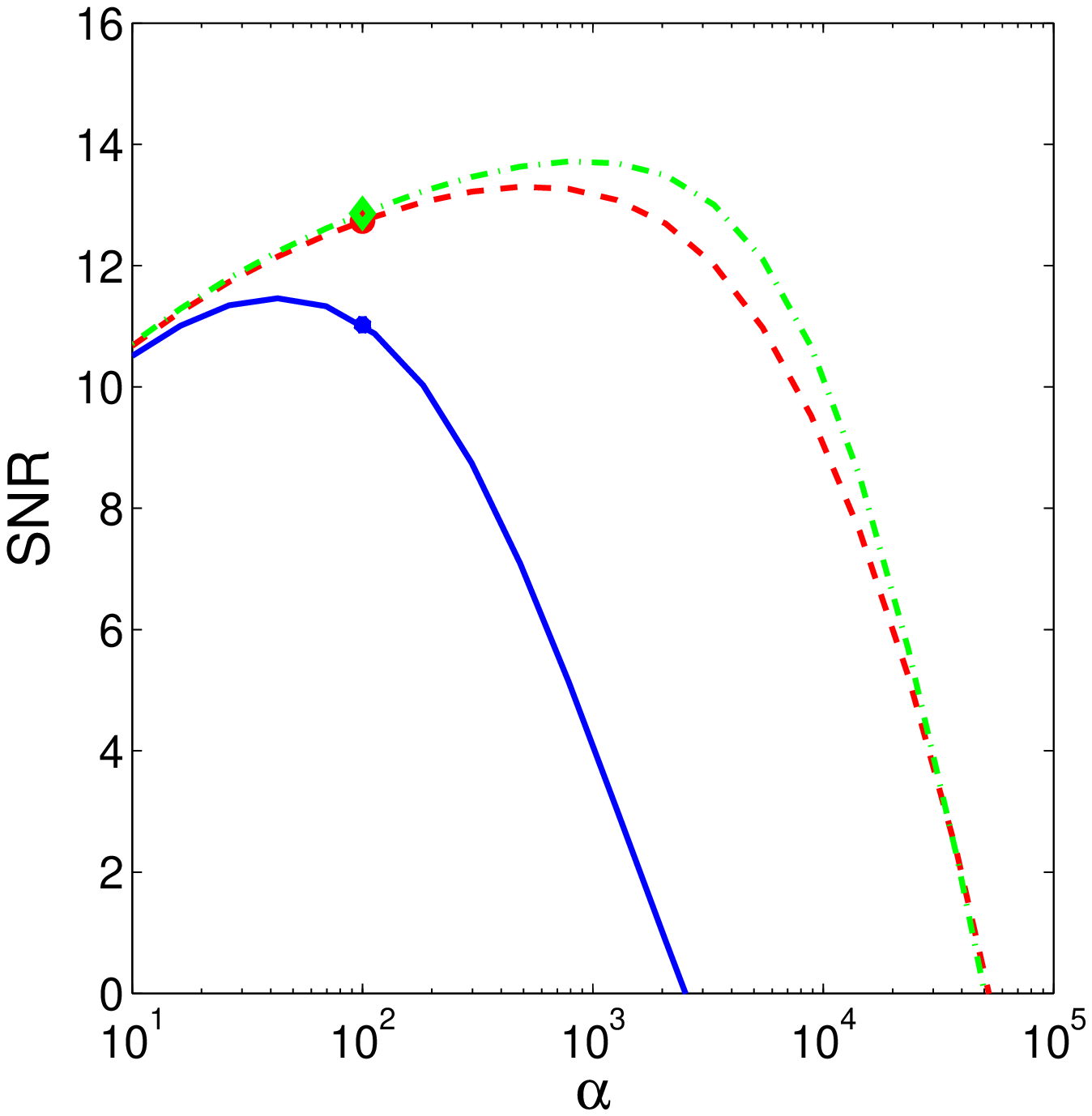,width=5.0cm}
        \small{$\sigma^2=5\times 10^{-4}$}
    \end{minipage}\\ \vspace{0.2cm}
    \begin{minipage}[c]{5.0cm}
        \centering
        \epsfig{figure=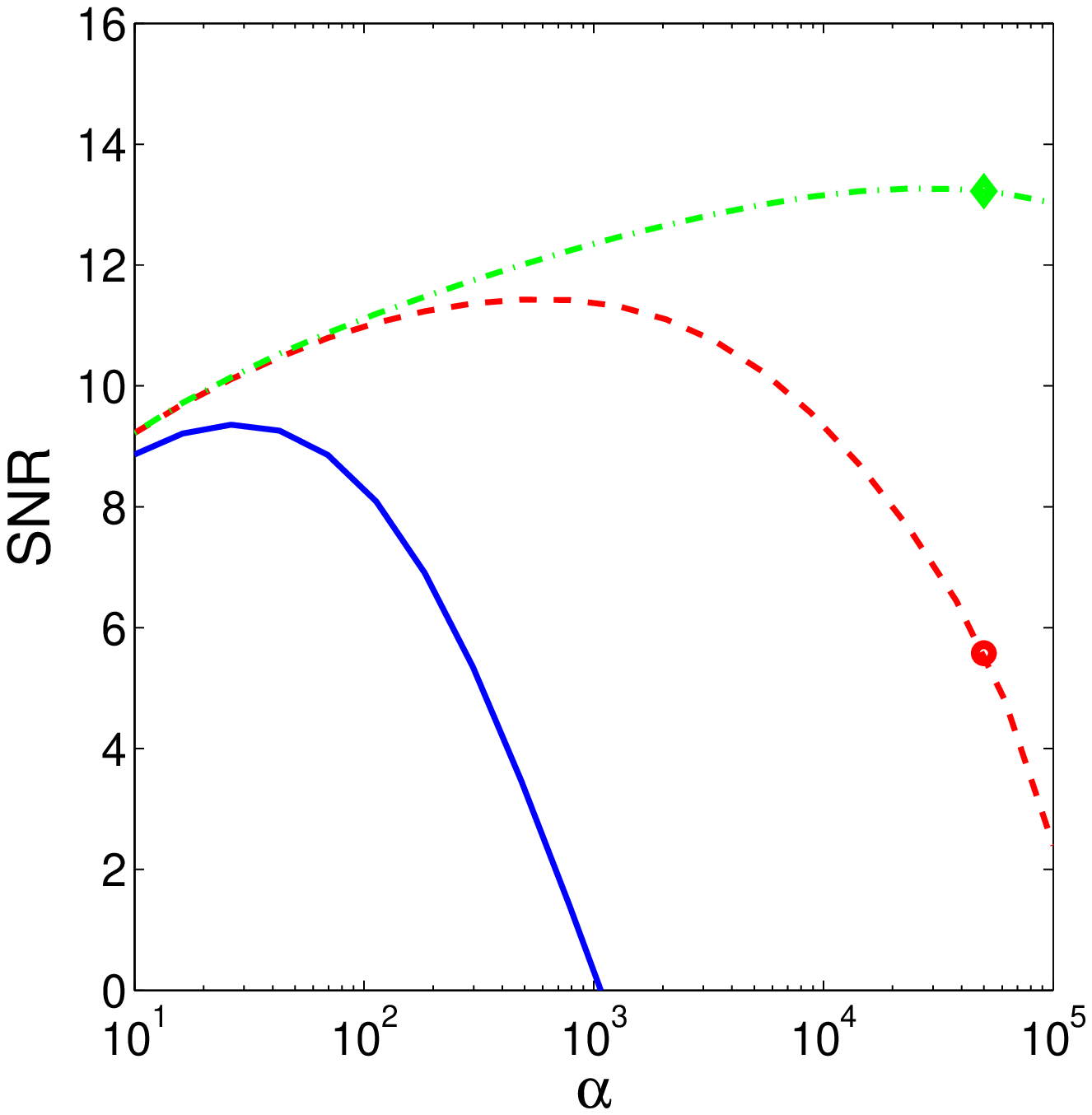,width=5.0cm}
        \small{$\sigma^2=10^{-6}$}
    \end{minipage}
    \begin{minipage}[c]{5.0cm}
        \centering
        \epsfig{figure=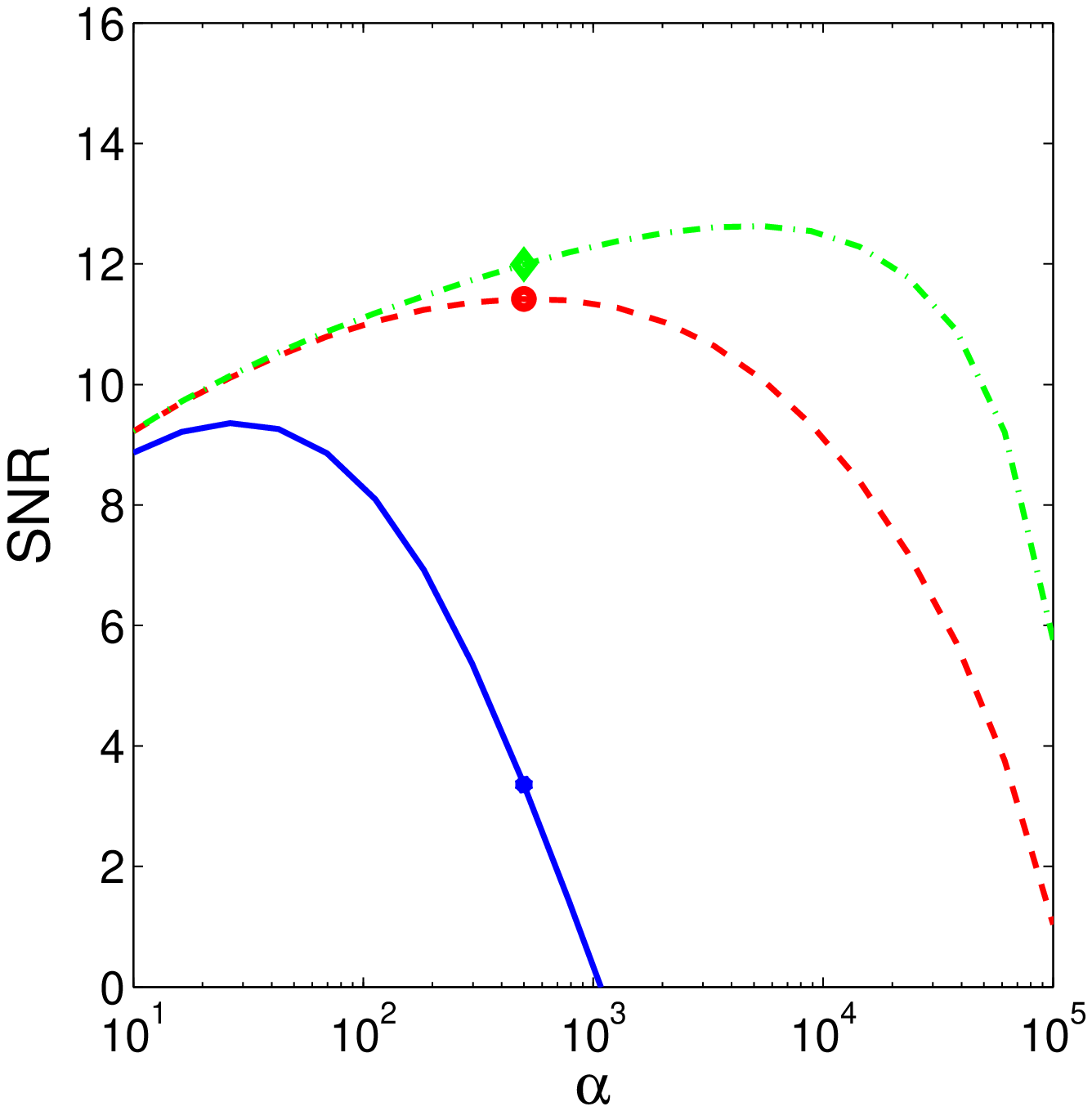,width=5.0cm}
        \small{$\sigma^2=10^{-4}$}
   \end{minipage}
    \begin{minipage}[c]{5.0cm}
        \centering
        \epsfig{figure=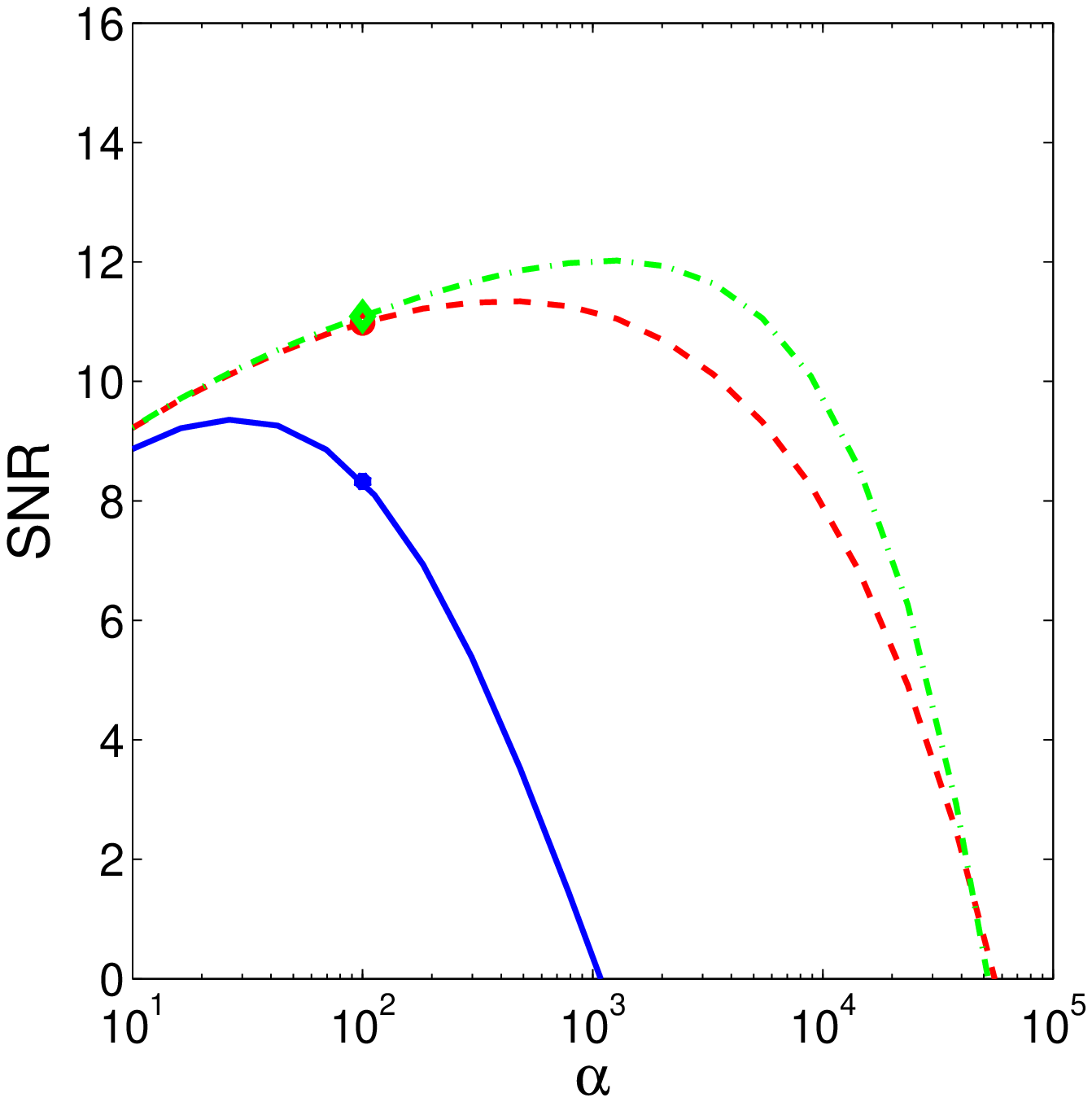,width=5.0cm}
        \small{$\sigma^2=5\times 10^{-4}$}
    \end{minipage}
    \end{center}
\end{center}
\caption{The SNR versus the regularization parameter $\alpha$ for different BCs and different vaues of the noise. The first row contains  the results for PSF with ${\tt hsize}=16$ and $\delta=5$ (the solid, dashed, and  dashdot lines denote Periodic, Reflective and Anti-reflective BCs, respectively; the star, circle, and diamond are related to $\alpha=0.05/\sigma^2$); the second row contains the results for PSF with ${\tt hsize}=22$ and $\delta=7$.} \label{fig:snr-c}
\end{figure}

Figure~\ref{fig:im-rs1} shows the computed solutions for the highest SNR varying $\alpha$, for the PSF a) (${\tt hsize}=16$ and $\delta=5$) at the noise level $\sigma^2=10^{-6}$.
The antireflective BCs provide the best restoration even if a slightly grater CPU time is required. Since increasing the accuracy of the boundary model a lower regularization is required, antireflective BCs reach the maximum SNR for a larger $\alpha$ with respect to the other BCs.  Comparing reflective and antireflective BCs at the same value of $\alpha$, also with a higher noise level ($\sigma^2=10^{-4}$), we note that not only antireflective BCs computes a better restoration, but they are also substantially more robust when varying $\alpha$, as shown in  Figure~\ref{fig:im-rs4}.

\begin{figure}
\begin{center}
    \begin{center}
    \begin{minipage}[c]{5.0cm}
        \centering
        \epsfig{figure=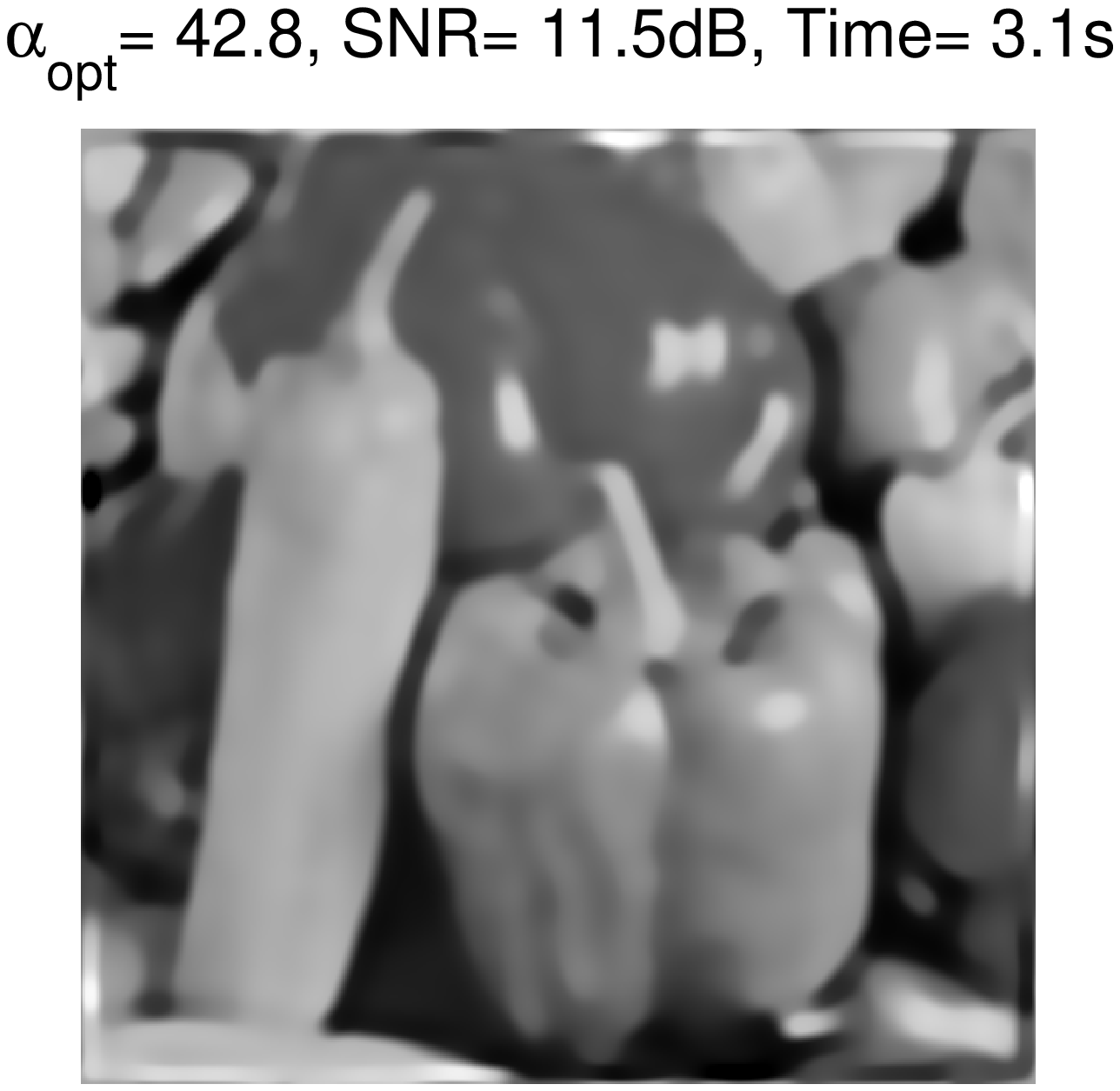,width=5.0cm}
        \small{Periodic}
    \end{minipage}
    \begin{minipage}[c]{5.0cm}
        \centering
        \epsfig{figure=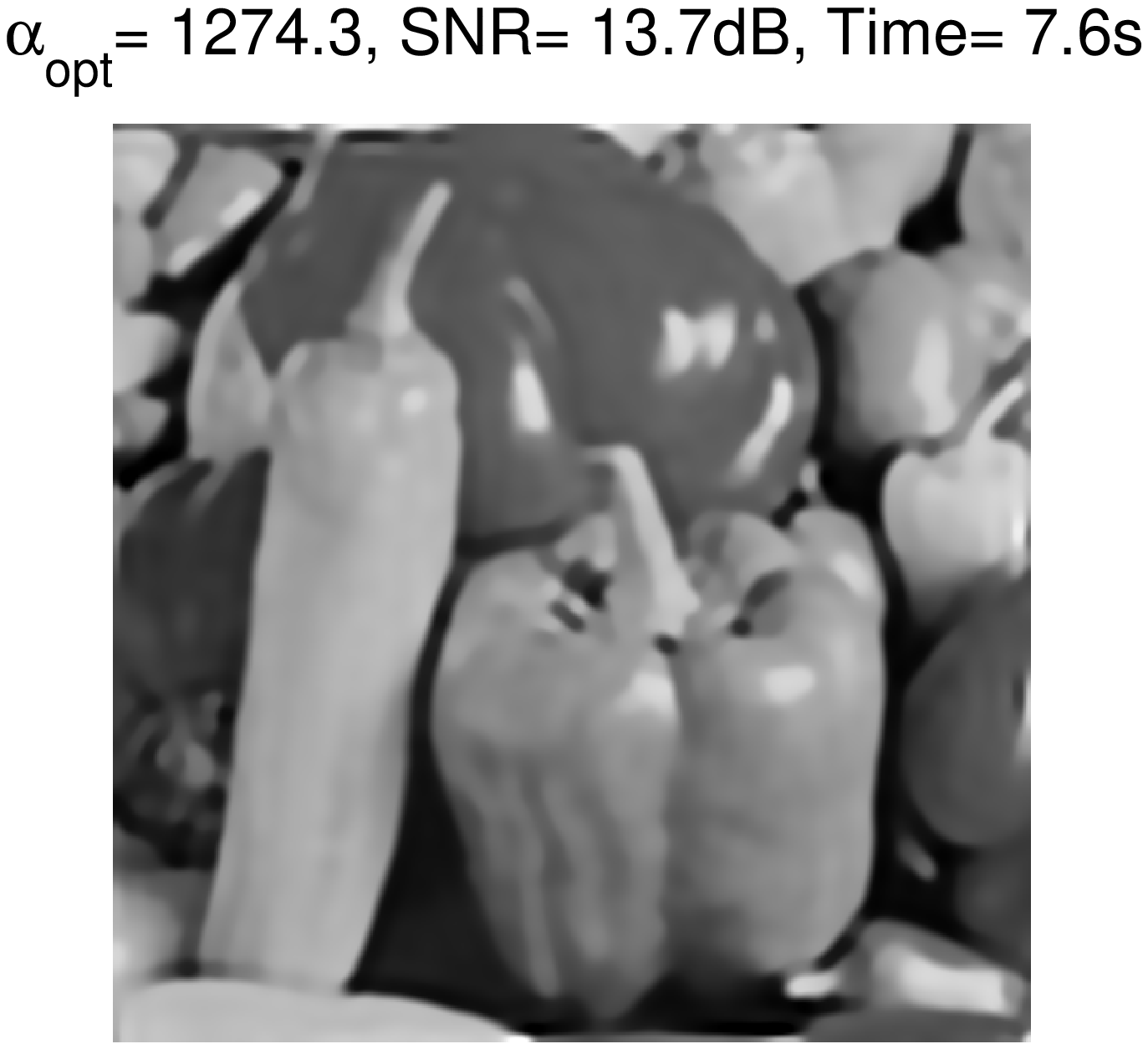,width=5.0cm}
        \small{Reflective}
   \end{minipage}
    \begin{minipage}[c]{5.0cm}
        \centering
        \epsfig{figure=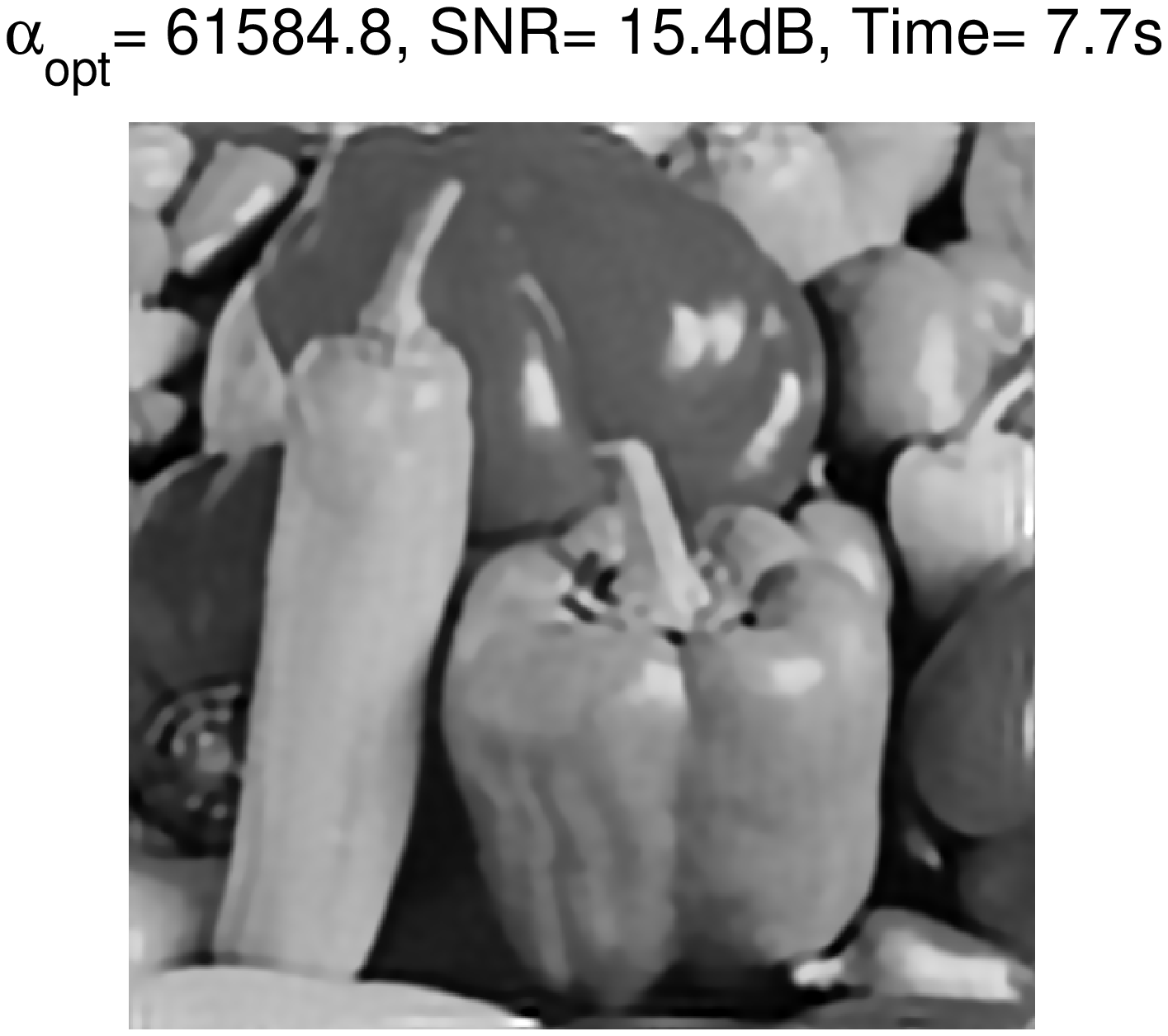,width=5.0cm}
        \small{Anti-reflective}
    \end{minipage} \\ \vspace{0.2cm}
    \begin{minipage}[c]{5.0cm}
        \centering
        \epsfig{figure=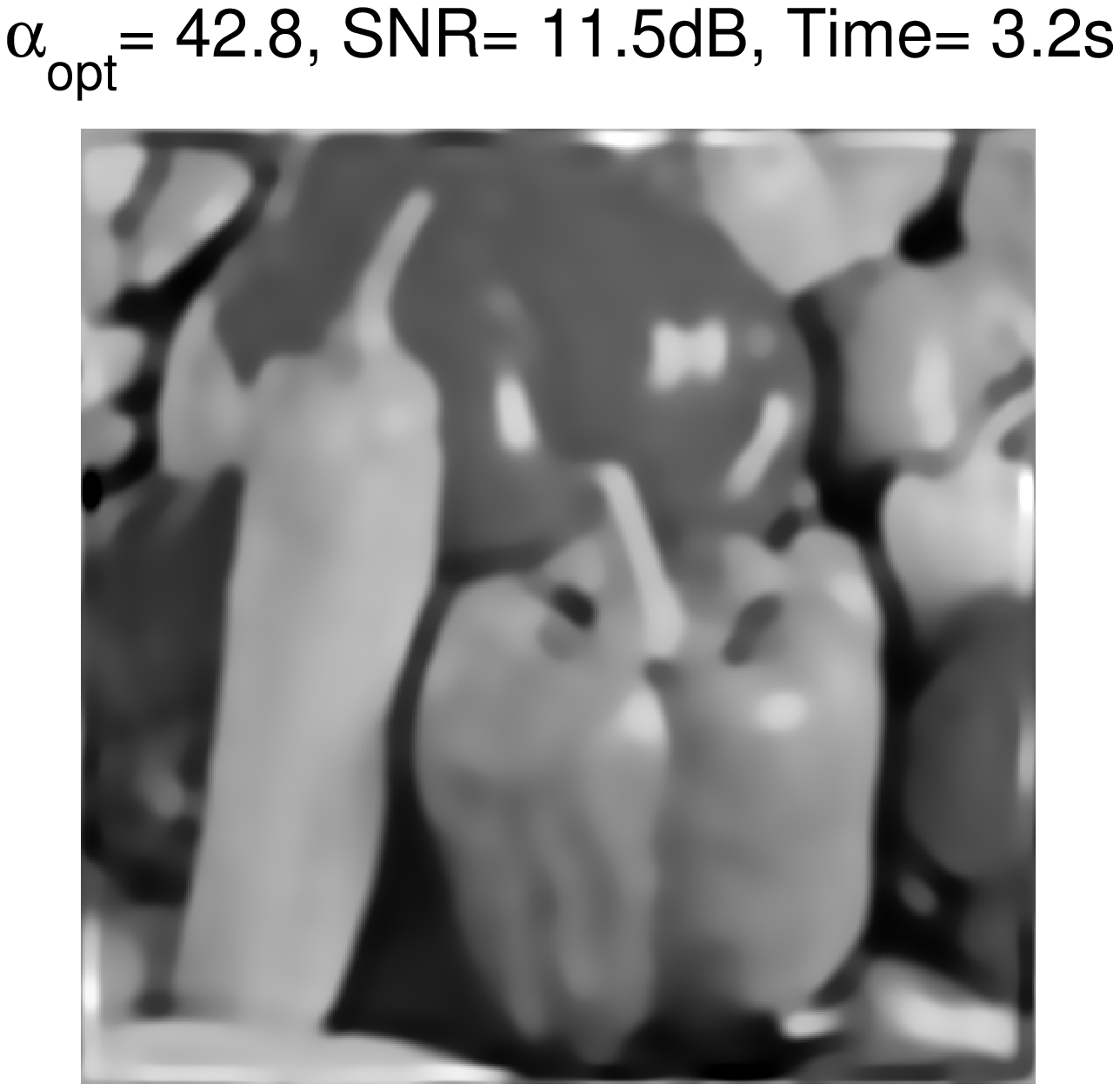,width=5.0cm}
        \small{Periodic}
    \end{minipage}
    \begin{minipage}[c]{5.0cm}
        \centering
        \epsfig{figure=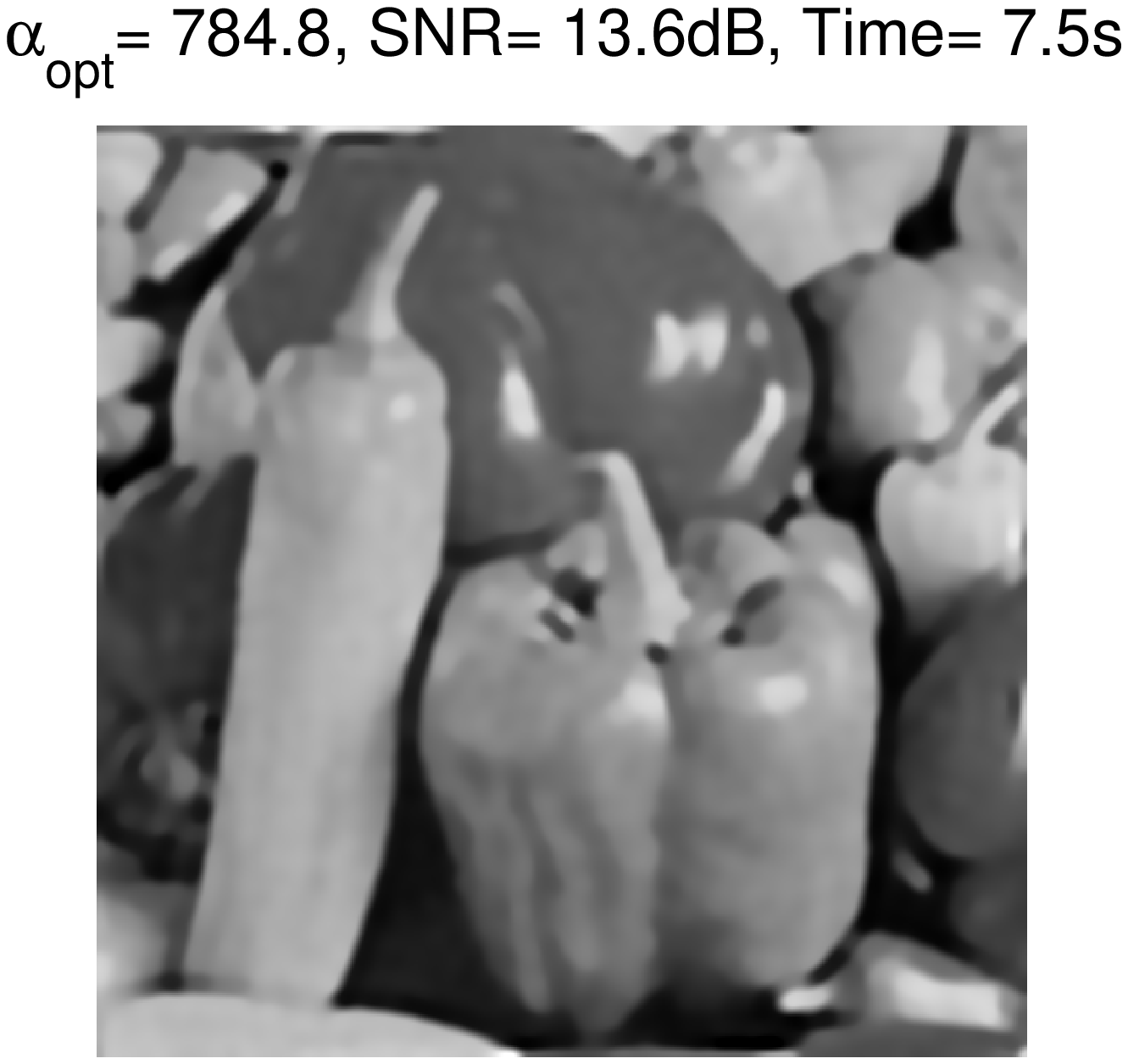,width=5.0cm}
        \small{Reflective}
   \end{minipage}
    \begin{minipage}[c]{5.0cm}
        \centering
        \epsfig{figure=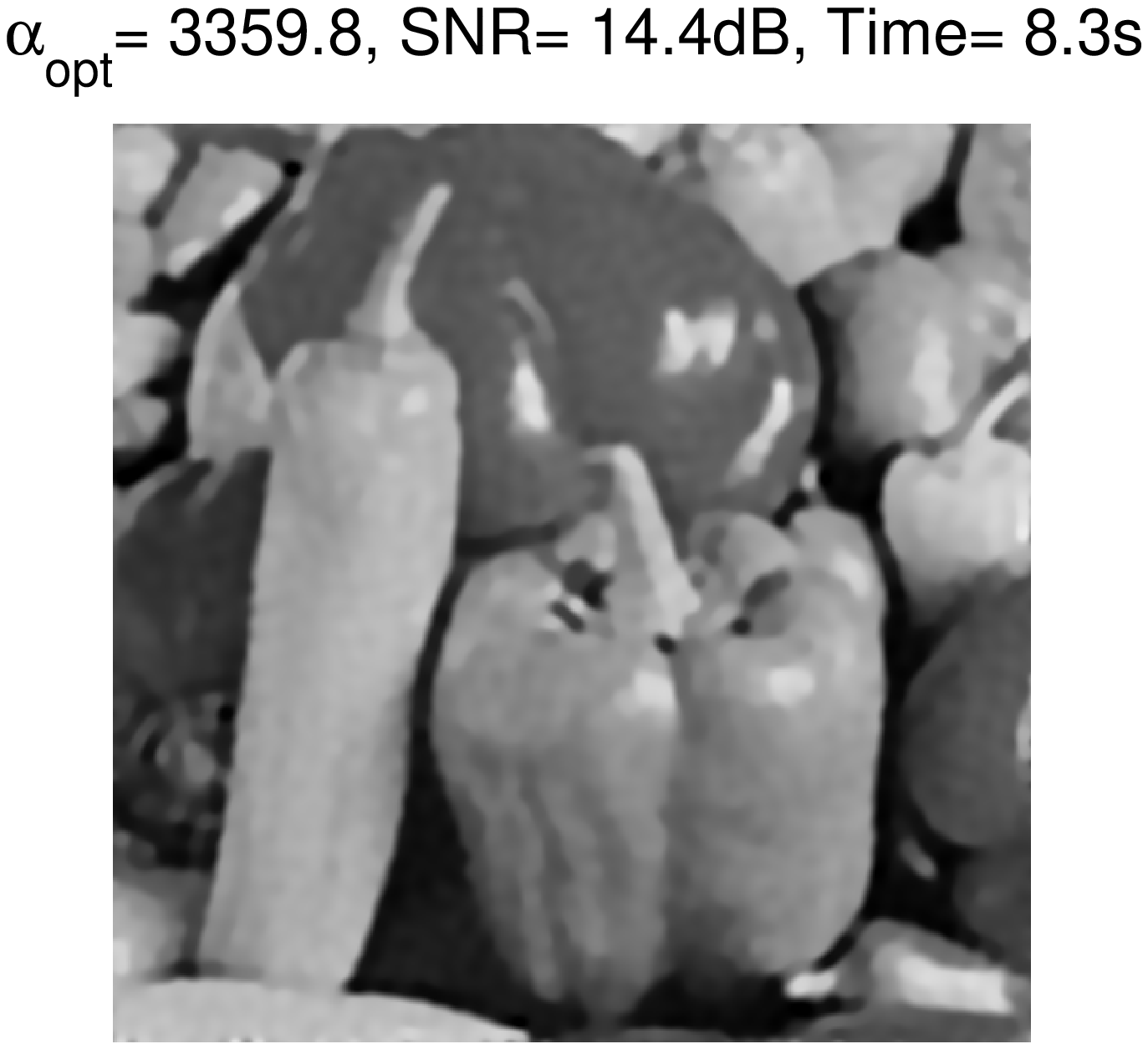,width=5.0cm}
        \small{Anti-reflective}
    \end{minipage}\\ \vspace{0.2cm}
    \begin{minipage}[c]{5.0cm}
        \centering
        \epsfig{figure=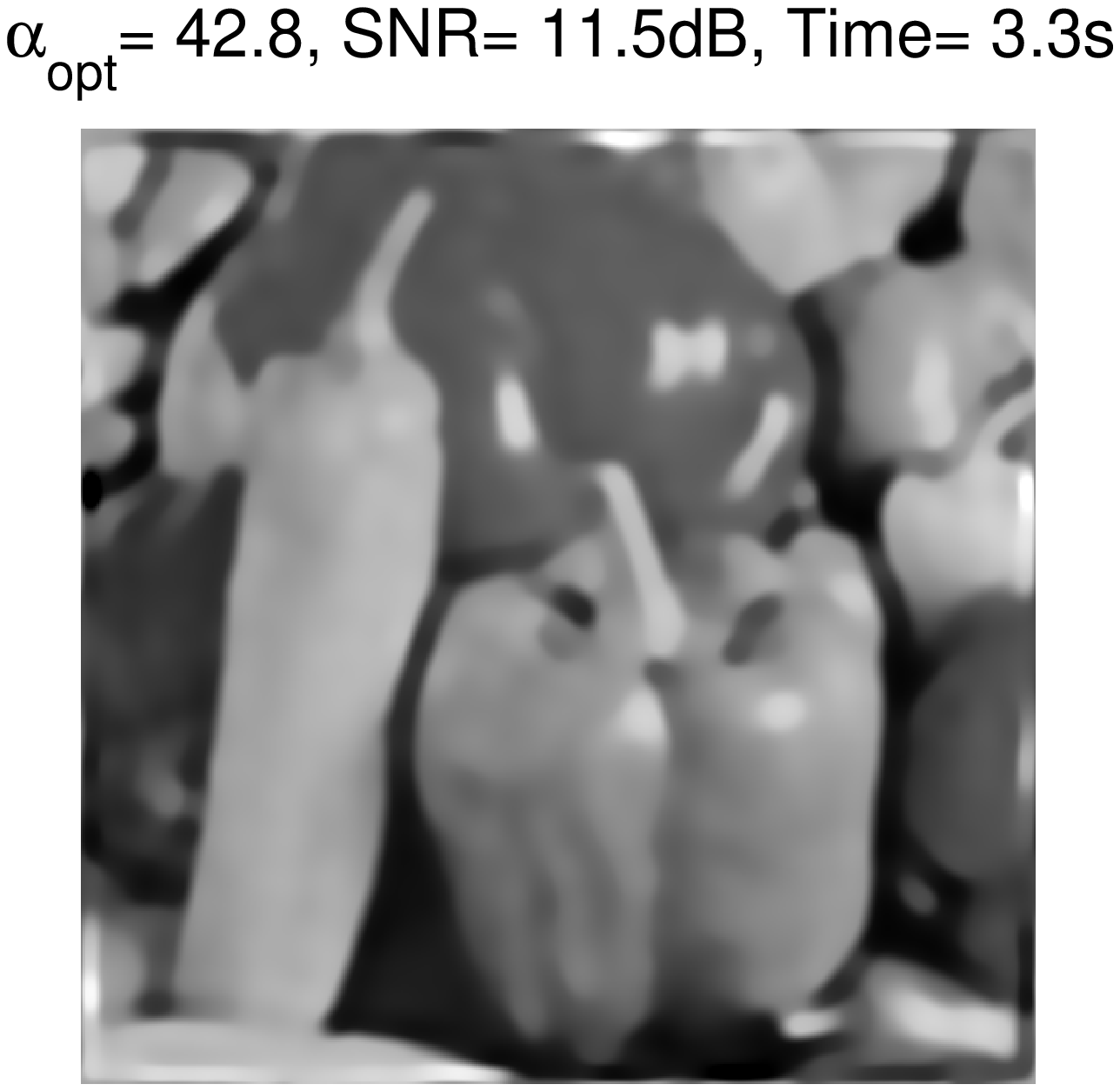,width=5.0cm}
        \small{Periodic}
    \end{minipage}
    \begin{minipage}[c]{5.0cm}
        \centering
        \epsfig{figure=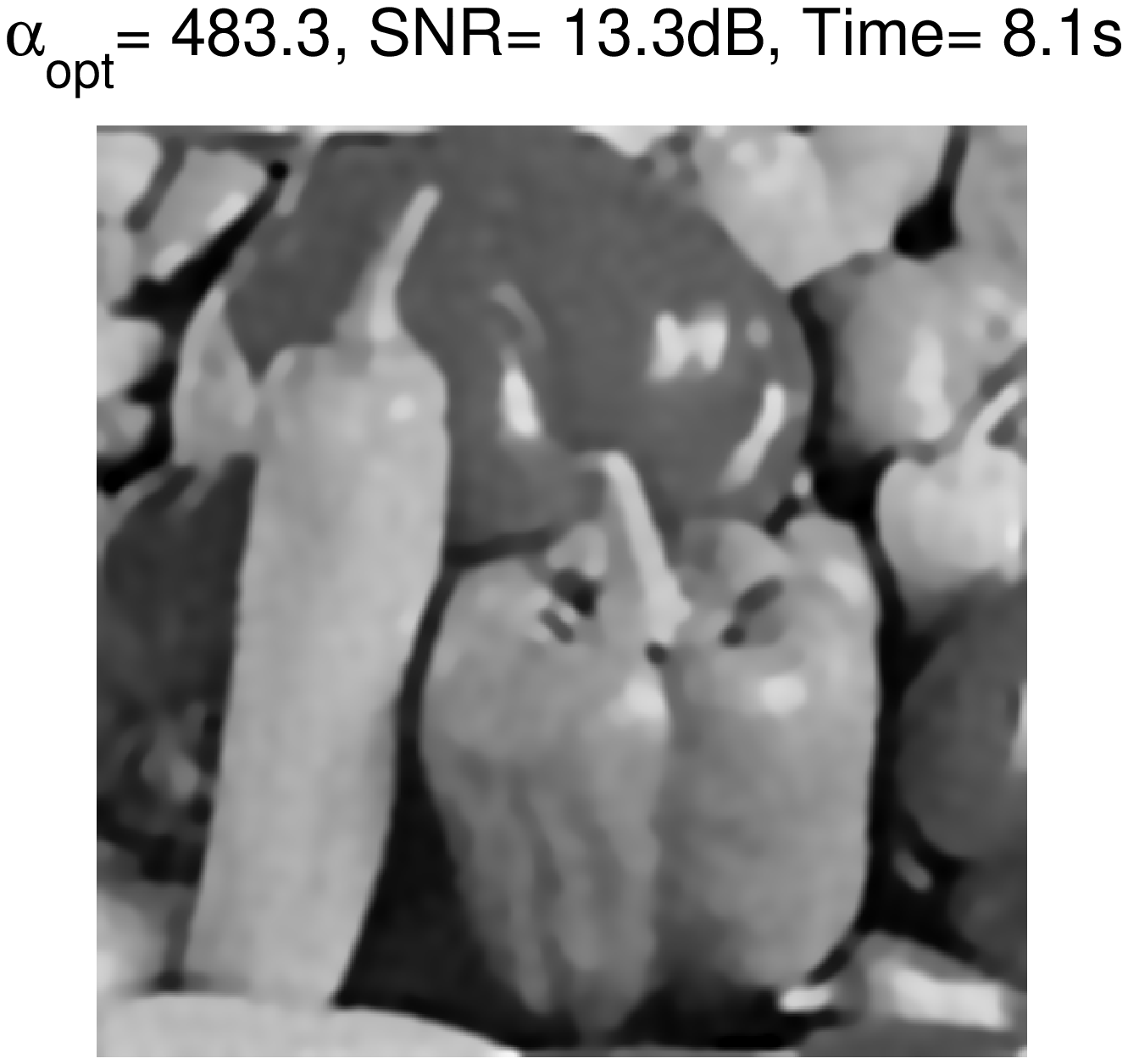,width=5.0cm}
        \small{Reflective}
   \end{minipage}
    \begin{minipage}[c]{5.0cm}
        \centering
        \epsfig{figure=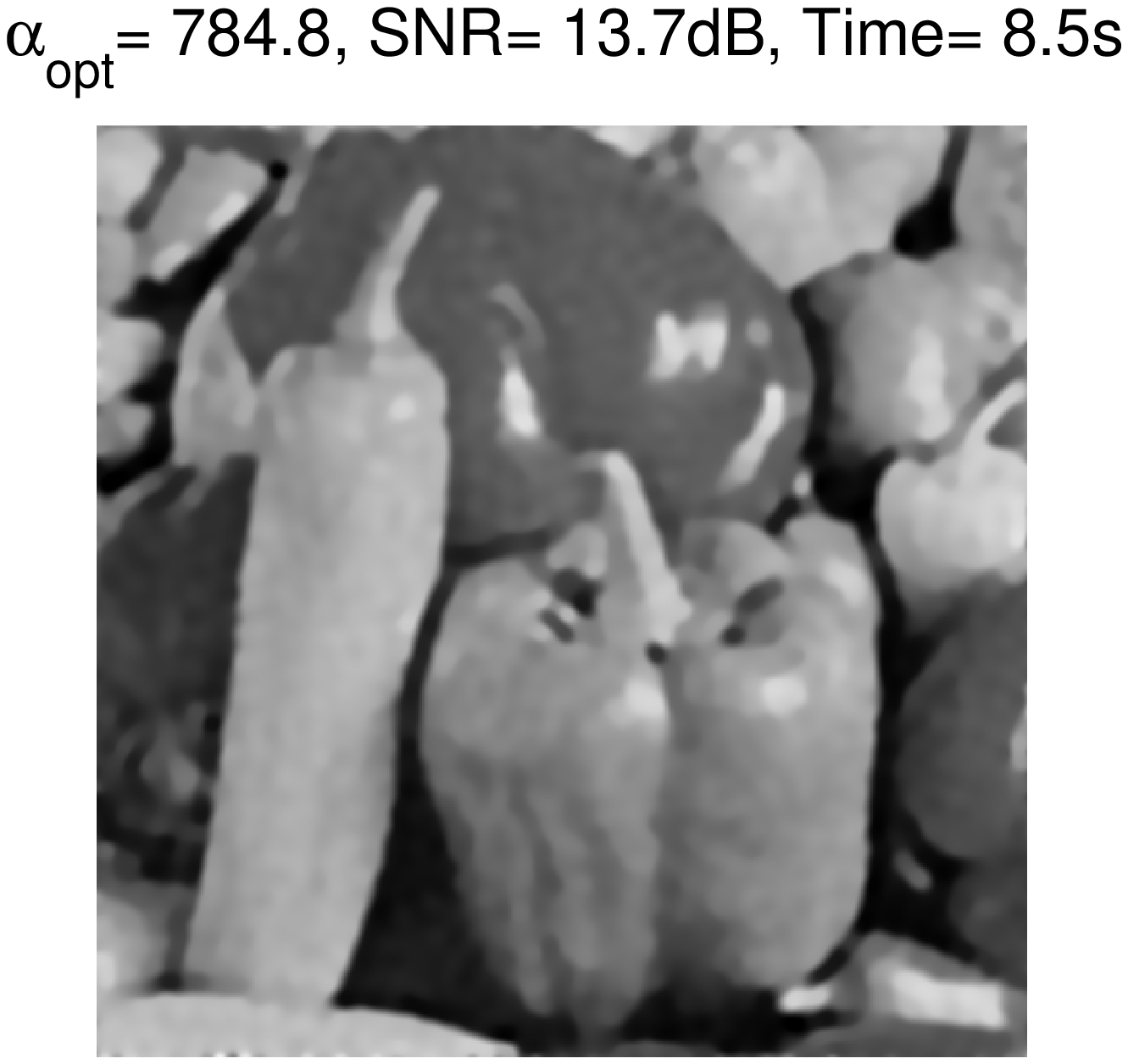,width=5.0cm}
        \small{Anti-reflective}
    \end{minipage}
    \end{center}
\end{center}
\caption{Best restorations for different BCs (PSF with ${\tt hsize}=16$ and $\delta=5$). The first row contains the results for $\sigma^2=10^{-6}$; the second row contains the results for $\sigma^2=10^{-4}$; the third row contains the results for $\sigma^2=5\times 10^{-4}$.} \label{fig:im-rs1}
\end{figure}

\begin{figure}[tb]
\begin{center}
    \begin{center}
    \begin{minipage}[c]{5.0cm}
        \centering
        \epsfig{figure=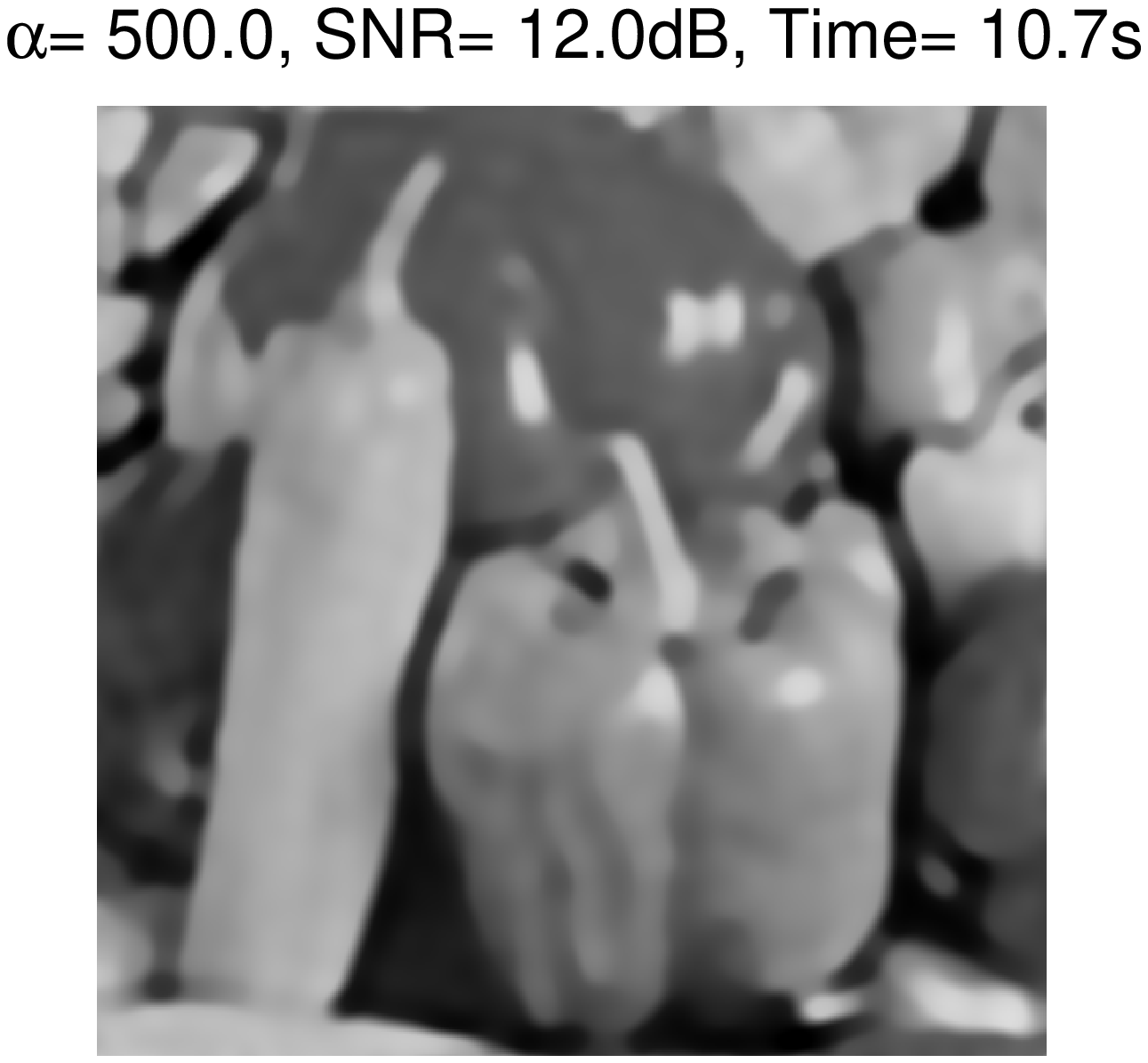,width=5.0cm}
        \small{Anti-reflective}
    \end{minipage}
    \begin{minipage}[c]{5.0cm}
        \centering
        \epsfig{figure=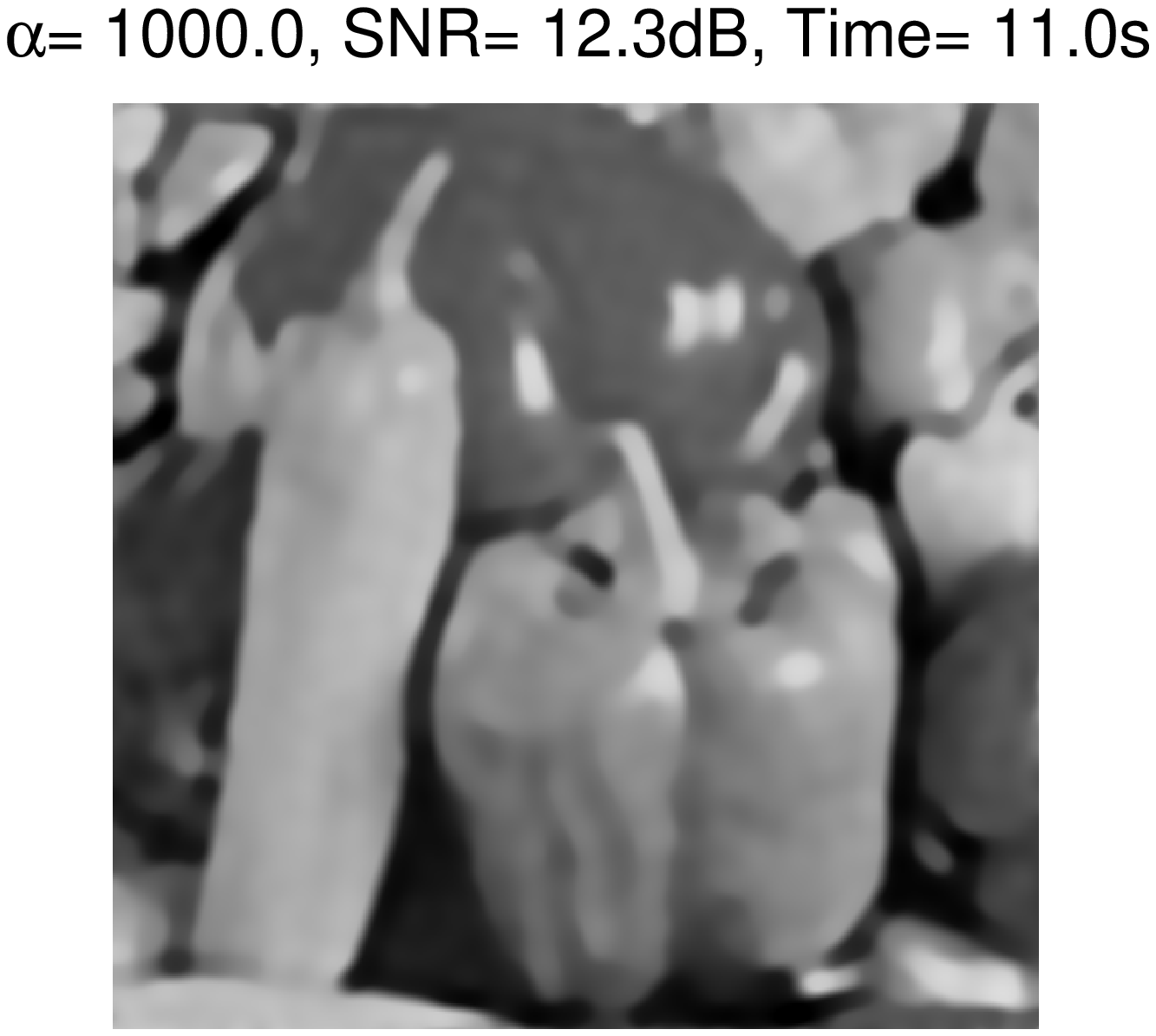,width=5.0cm}
        \small{Anti-reflective}
   \end{minipage}
    \begin{minipage}[c]{5.0cm}
        \centering
        \epsfig{figure=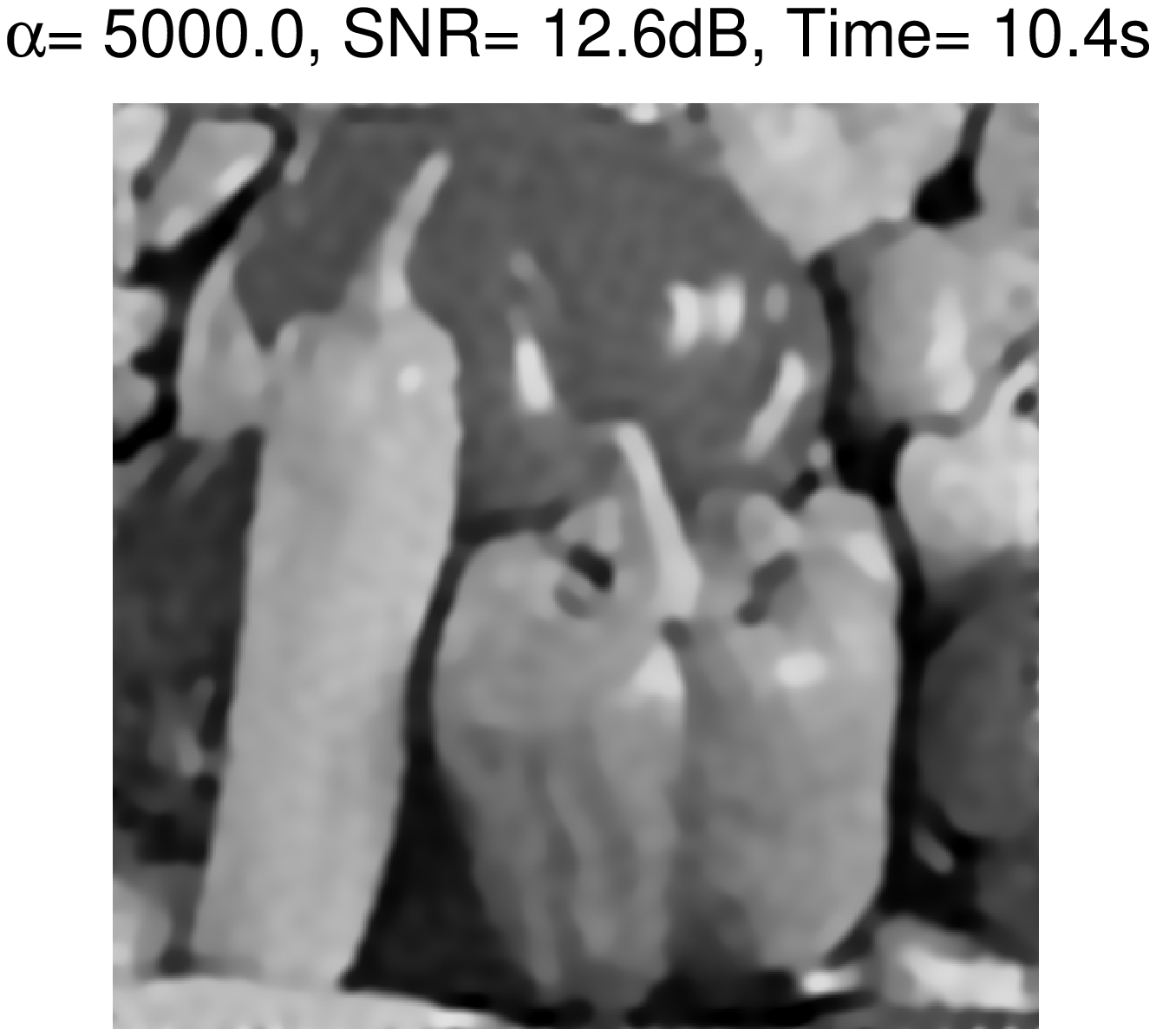,width=5.0cm}
        \small{Anti-reflective}
    \end{minipage} \\ \vspace{0.2cm}
    \begin{minipage}[c]{5.0cm}
        \centering
        \epsfig{figure=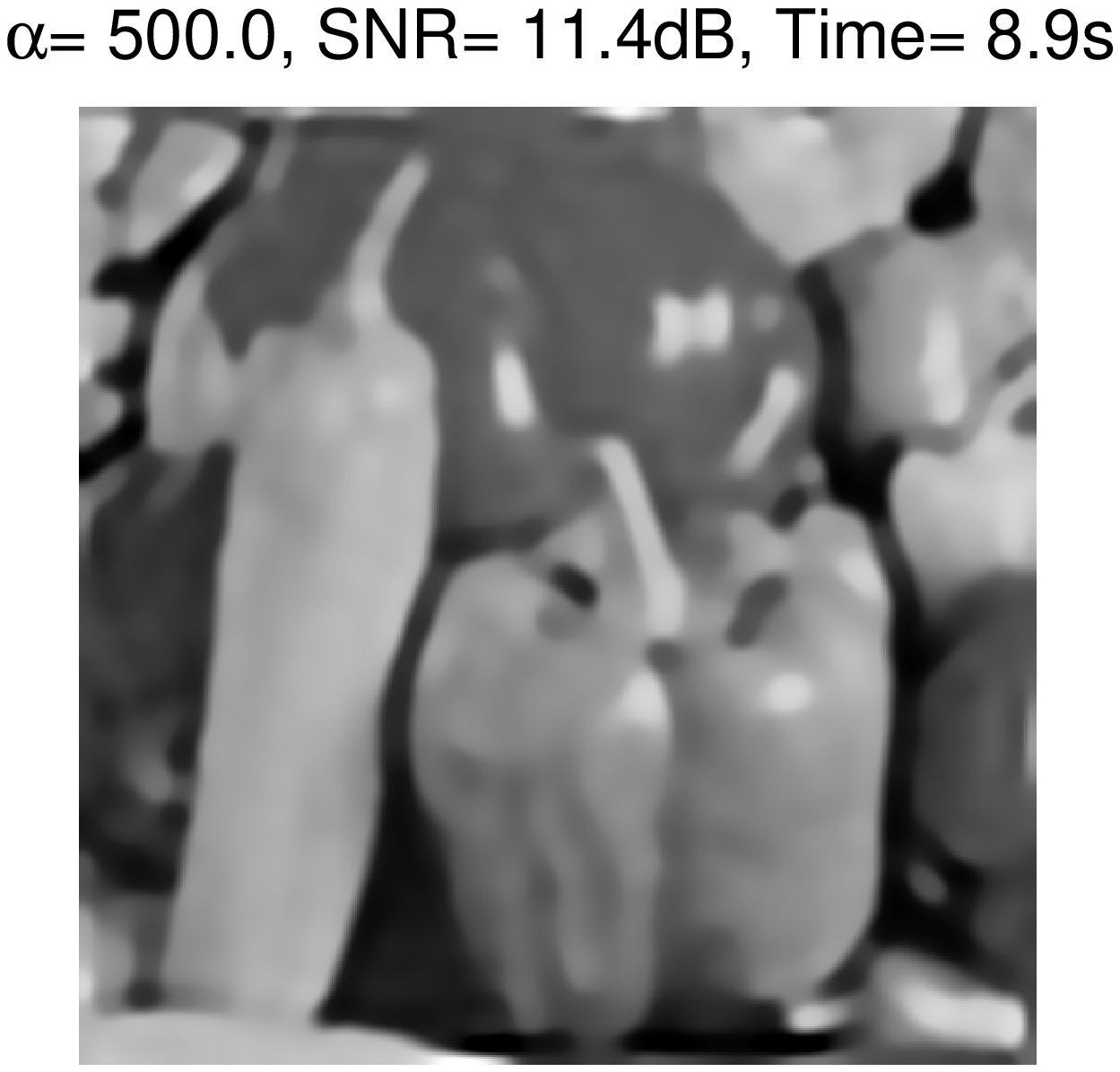,width=5.0cm}
        \small{Reflective}
    \end{minipage}
    \begin{minipage}[c]{5.0cm}
        \centering
        \epsfig{figure=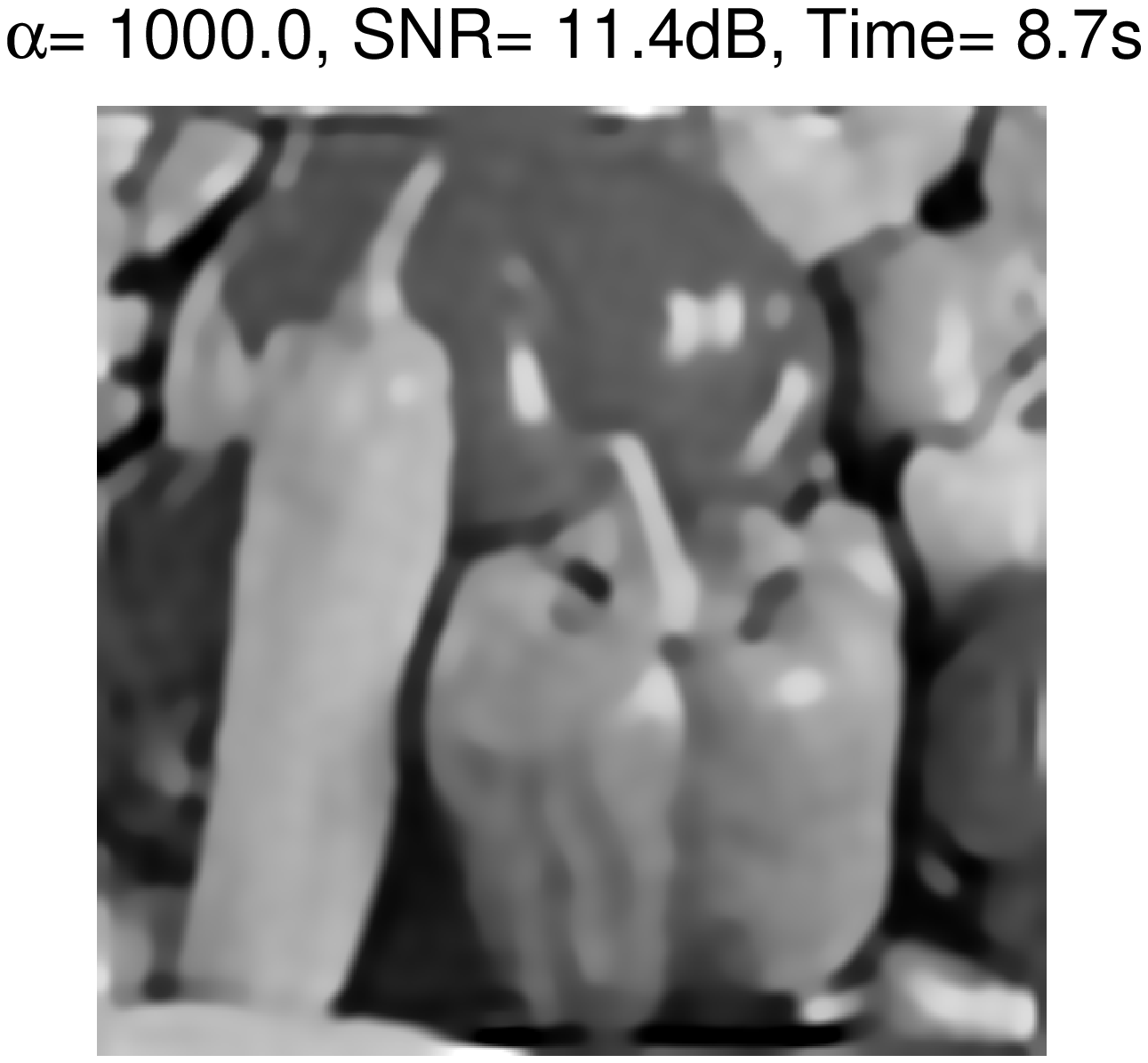,width=5.0cm}
        \small{Reflective}
   \end{minipage}
    \begin{minipage}[c]{5.0cm}
        \centering
        \epsfig{figure=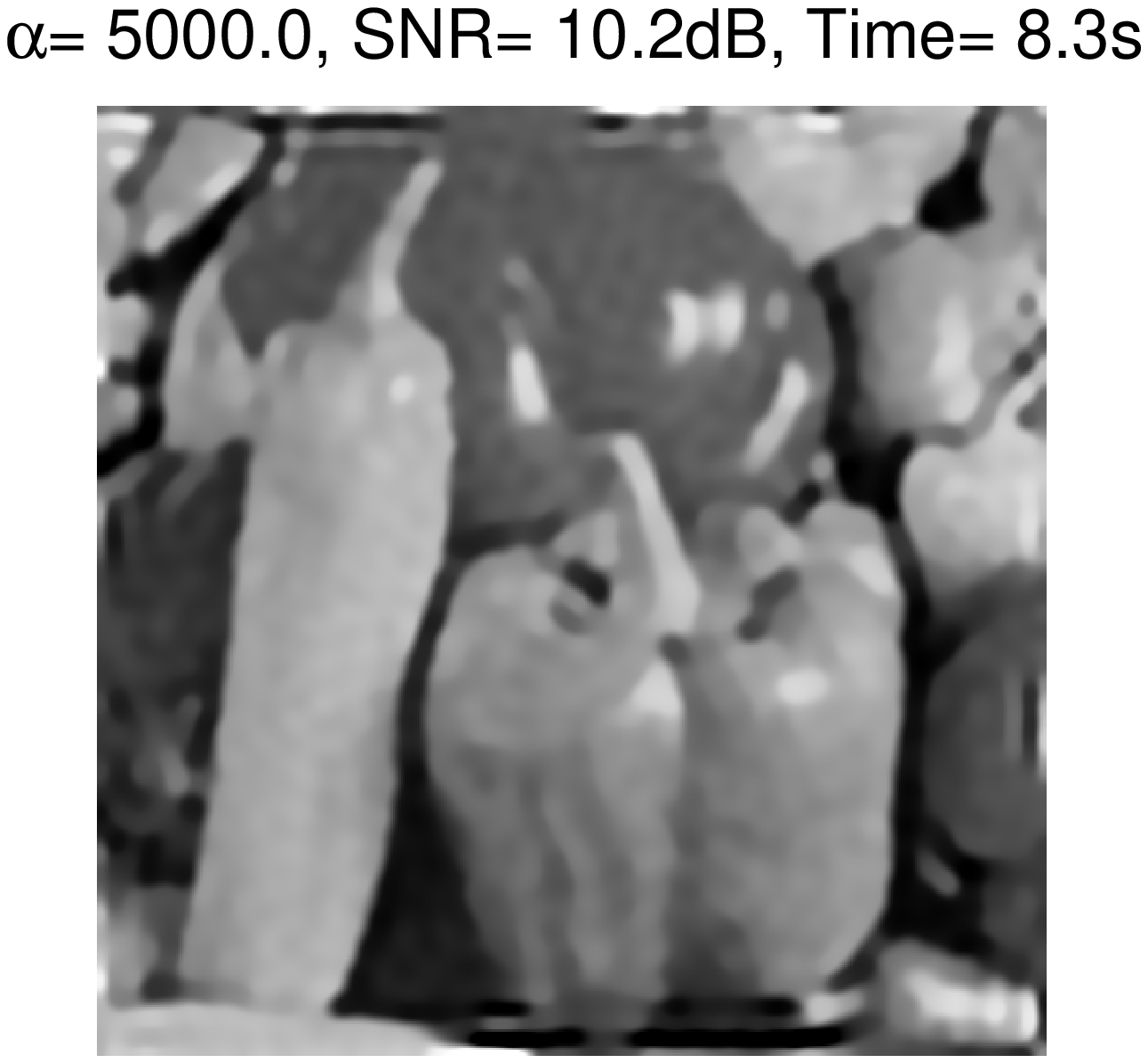,width=5.0cm}
        \small{Reflective}
    \end{minipage}
    \end{center}
\end{center}
\caption{Restored images for different BCs and different $\alpha$ (PSF with ${\tt hsize}=22$ and $\delta=7$ and $\sigma^2=10^{-4}$).} \label{fig:im-rs4}
\end{figure}

\subsection{Nonsymmetric PSF}
In this example we consider a nonsymmetric blur and a rectangular image. As already observed, the linear system \eqref{subp2-op-d} can not directly solved by fast trigonometric transform in the case of reflective and antireflective BCs, while periodic BCs do not provide accurate restorations as confirmed by the previous numerical example.
Consequently, as previously suggested, in this example we apply the second discretization strategy based on the enlargement of the domain.

Figure \ref{fig:im-tb2} shows the true image, the nonsymmetric PSF, and the blurred image without noise.
Figure \ref{fig:snr-c2} shows the SNR versus the regularization parameter $\alpha$ for different BCs and different noise levels. We note that the antireflective extension does not produce a large improvement in the SNR like in the previous example with the discrete algorithm based on the BCs model. However, around the maximum SNR, the antireflective extension provides a slightly better restoration with a lower CPU time with respect to the reflective extension, cf. Figure \ref{fig:im-rs5}.

\begin{figure}[tb]
\begin{center}
    \begin{center}
    \begin{minipage}[c]{5.0cm}
        \centering
        \epsfig{figure=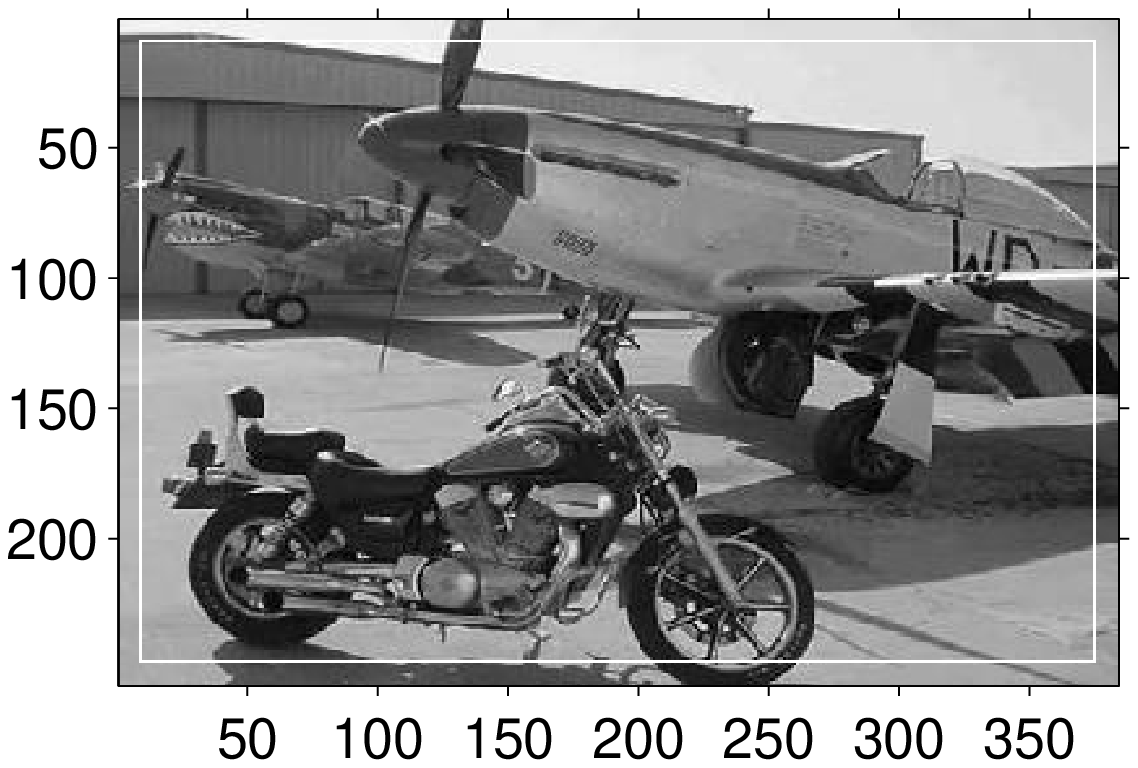,width=5.0cm}
        \small{True image}
    \end{minipage}
    \begin{minipage}[c]{5.0cm}
        \centering
        \epsfig{figure=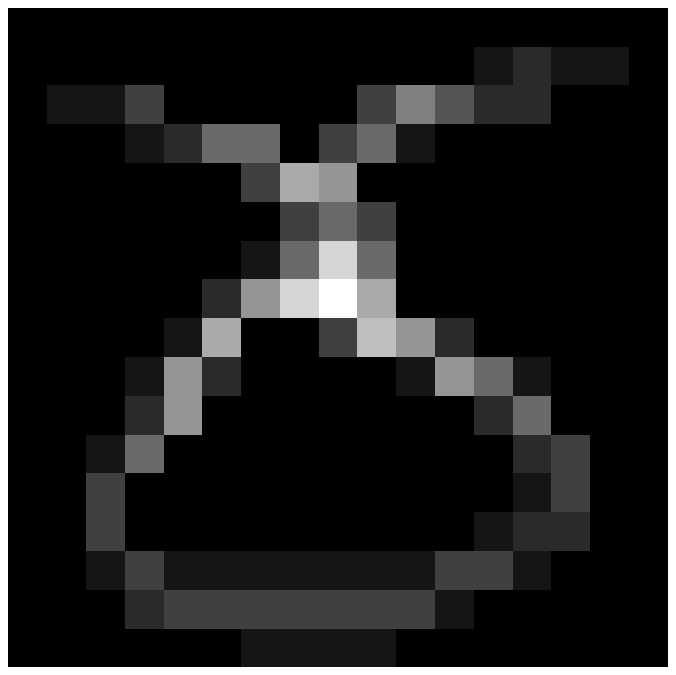,width=5.0cm}
        \small{PSF}
   \end{minipage}
    \begin{minipage}[c]{5.0cm}
        \centering
        \epsfig{figure=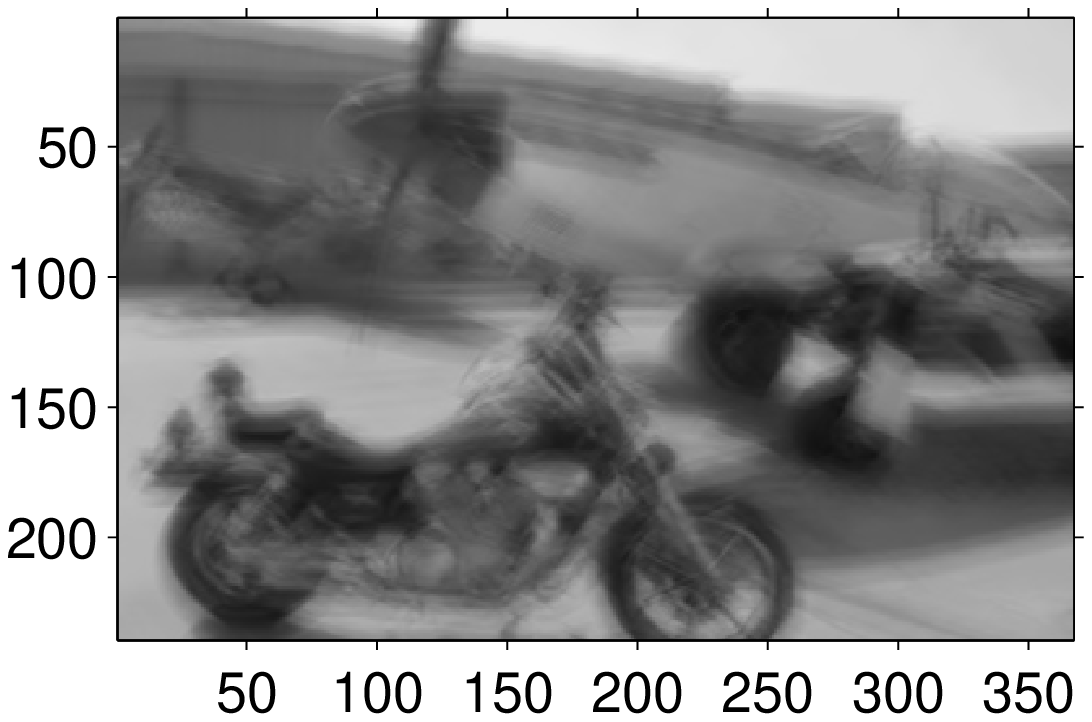,width=5.0cm}
        \small{Blurred image}
    \end{minipage}
    \end{center}
\end{center}
\caption{True image, PSF, and blurred image without noise.} \label{fig:im-tb2}
\end{figure}

\begin{figure}[tb]
\begin{center}
    \begin{center}
    \begin{minipage}[c]{5.0cm}
        \centering
        \epsfig{figure=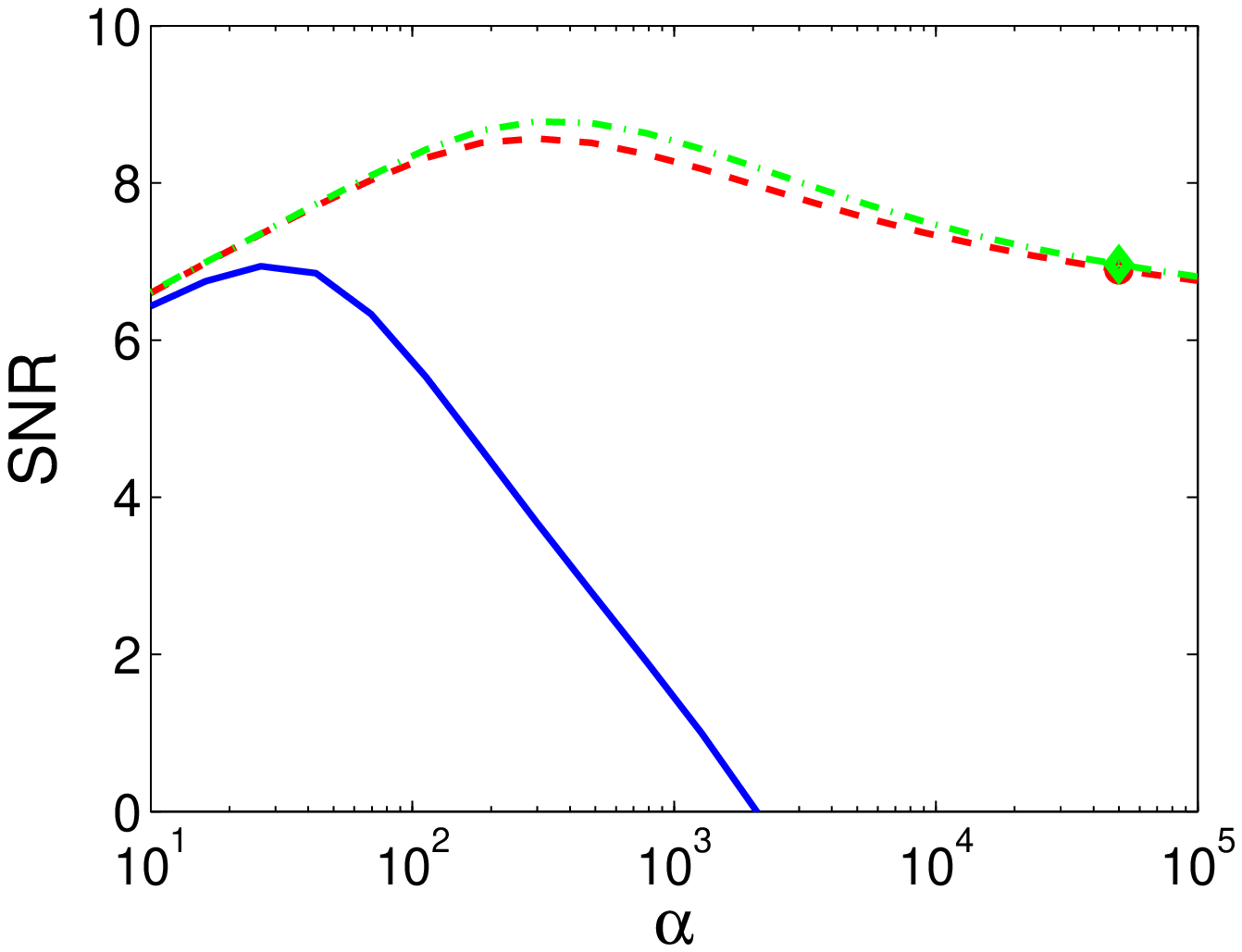,width=5.0cm}
        \small{$\sigma^2=10^{-6}$}
    \end{minipage}
    \begin{minipage}[c]{5.0cm}
        \centering
        \epsfig{figure=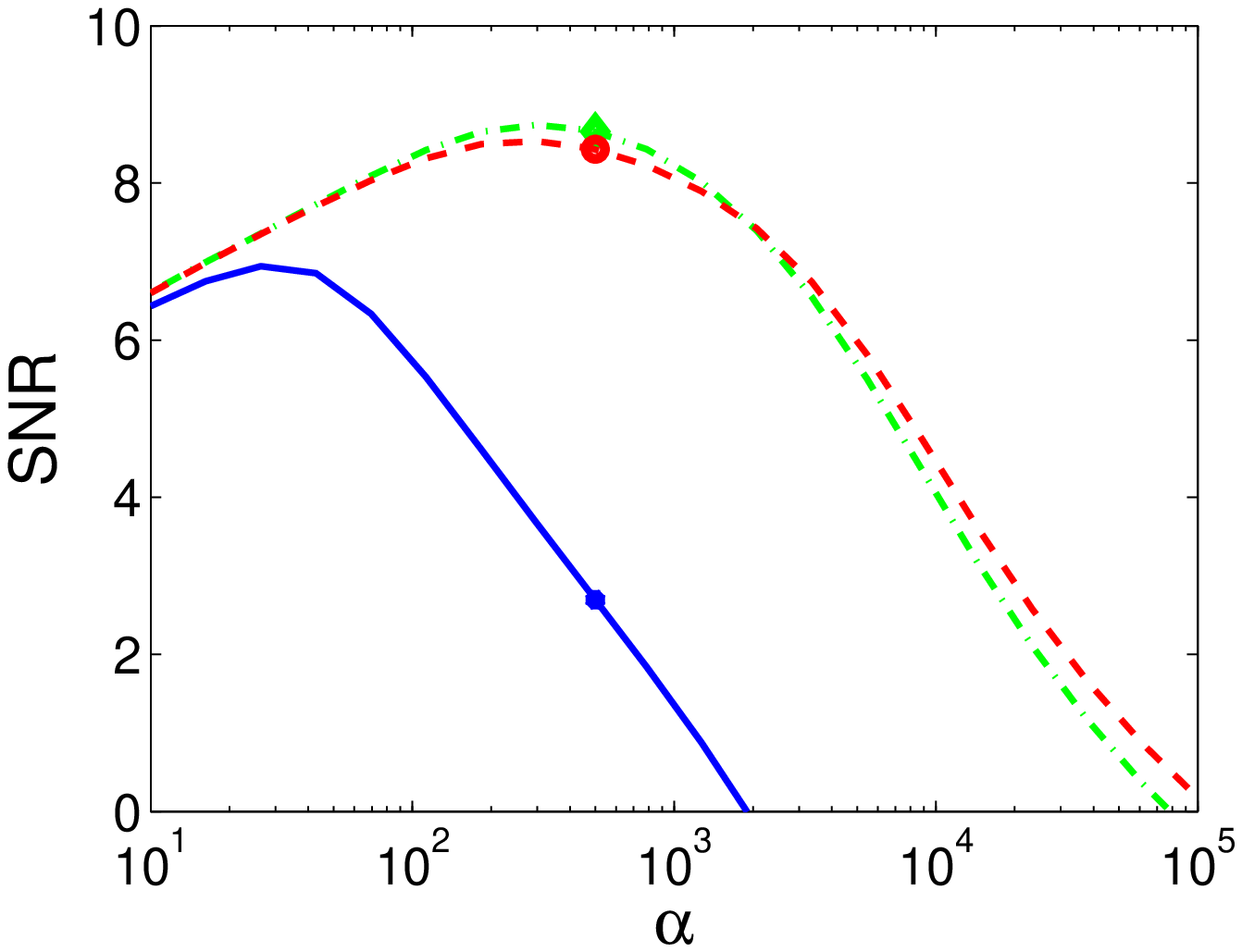,width=5.0cm}
        \small{$\sigma^2=10^{-4}$}
   \end{minipage}
    \begin{minipage}[c]{5.0cm}
        \centering
        \epsfig{figure=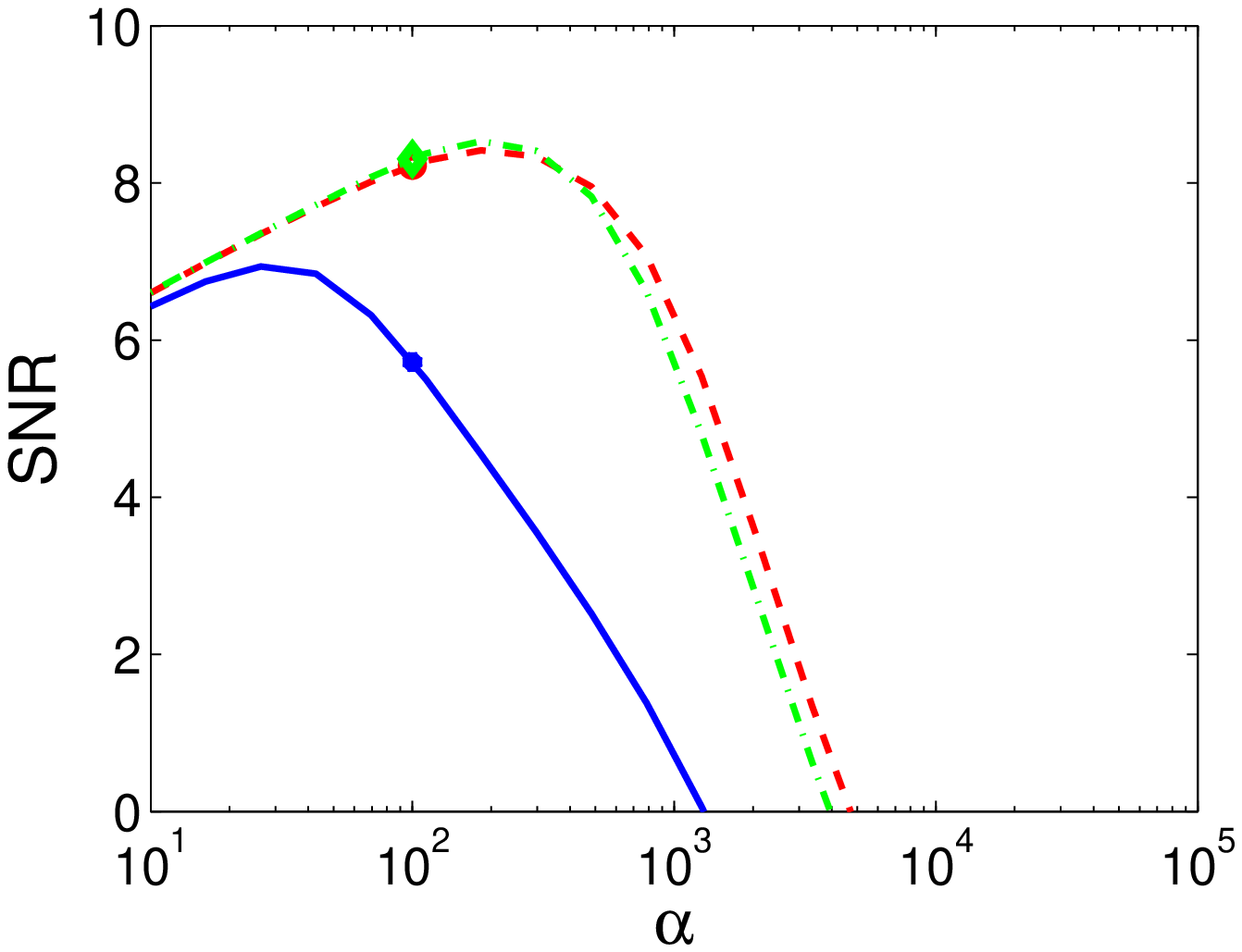,width=5.0cm}
        \small{$\sigma^2=5\times 10^{-4}$}
    \end{minipage}
    \end{center}
\end{center}
\caption{The SNR versus the regularization parameter $\alpha$ for different values of the noise (the solid, dashed, and  dashdot lines denote Periodic, Symmetric Extension, and Anti-symmetric Extension, respectively; the star, circle, and diamond are related to $\alpha=0.05/\sigma^2$).} \label{fig:snr-c2}
\end{figure}

\begin{figure}[tb]
\begin{center}
    \begin{center}
    \begin{minipage}[c]{5.0cm}
        \centering
        \epsfig{figure=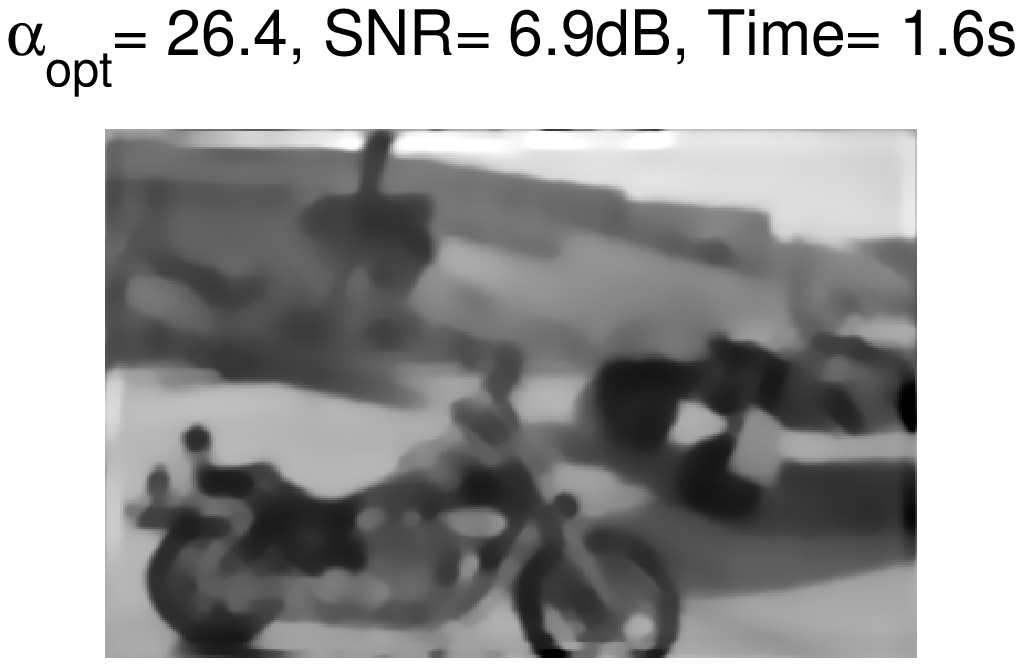,width=5.0cm}
        \small{Periodic}
    \end{minipage}
    \begin{minipage}[c]{5.0cm}
        \centering
        \epsfig{figure=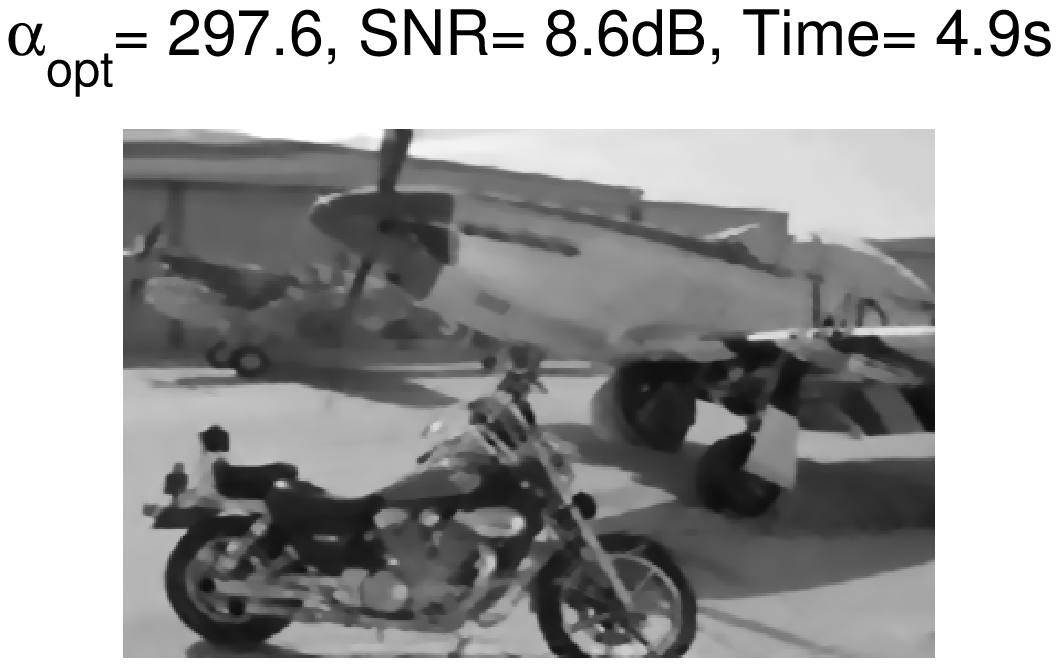,width=5.0cm}
        \small{Symmetric Extension}
   \end{minipage}
    \begin{minipage}[c]{5.0cm}
        \centering
        \epsfig{figure=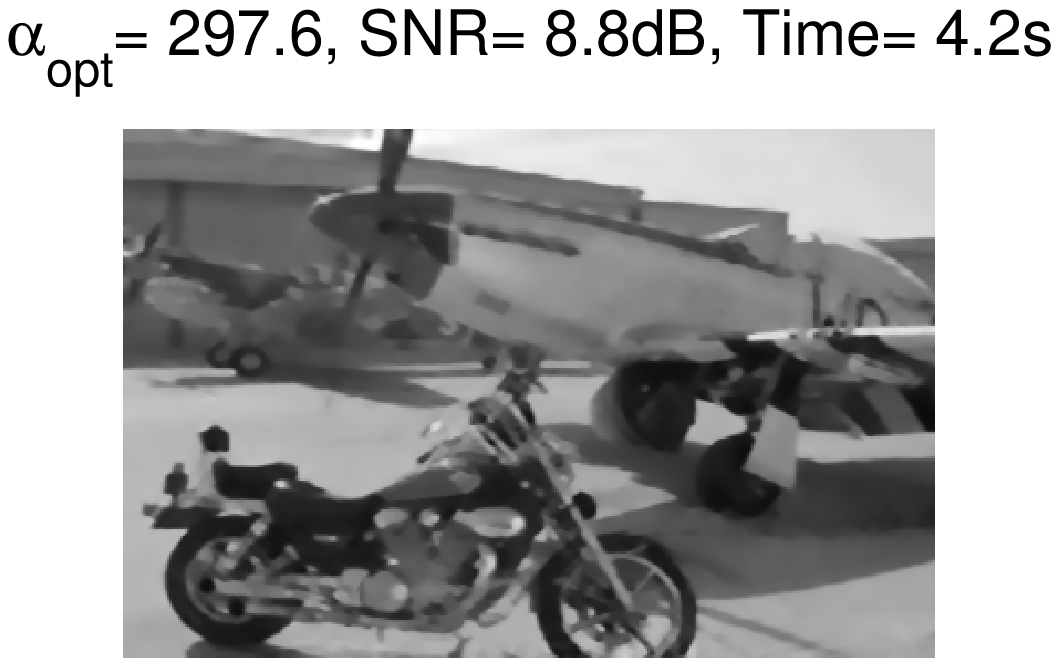,width=5.0cm}
        \small{Anti-symmetric Extension}
    \end{minipage} \\ \vspace{0.2cm}
    \begin{minipage}[c]{5.0cm}
        \centering
        \epsfig{figure=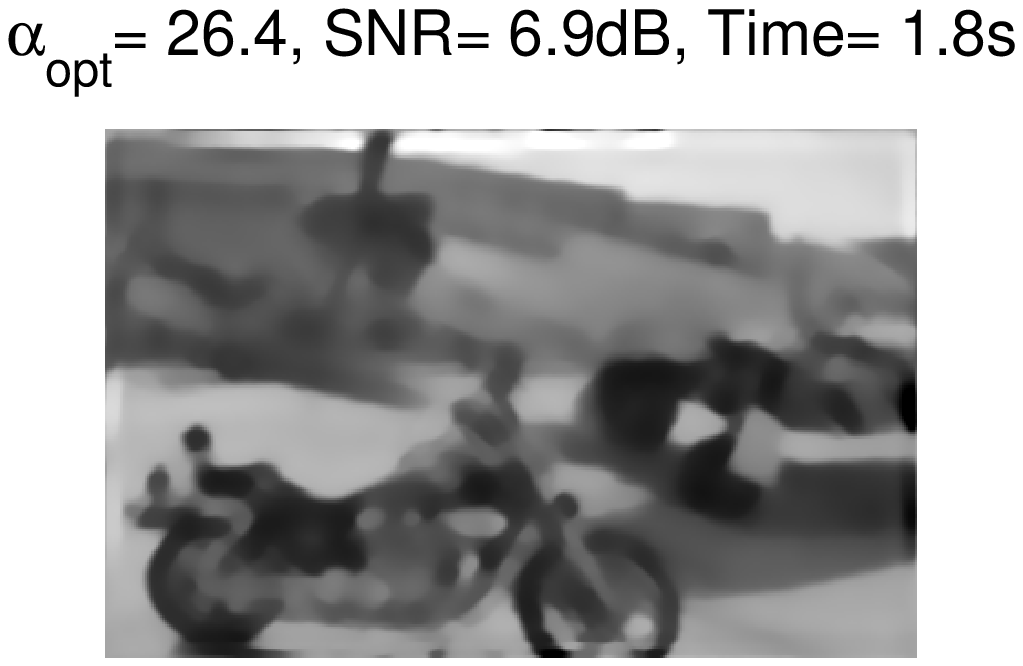,width=5.0cm}
        \small{Periodic}
    \end{minipage}
    \begin{minipage}[c]{5.0cm}
        \centering
        \epsfig{figure=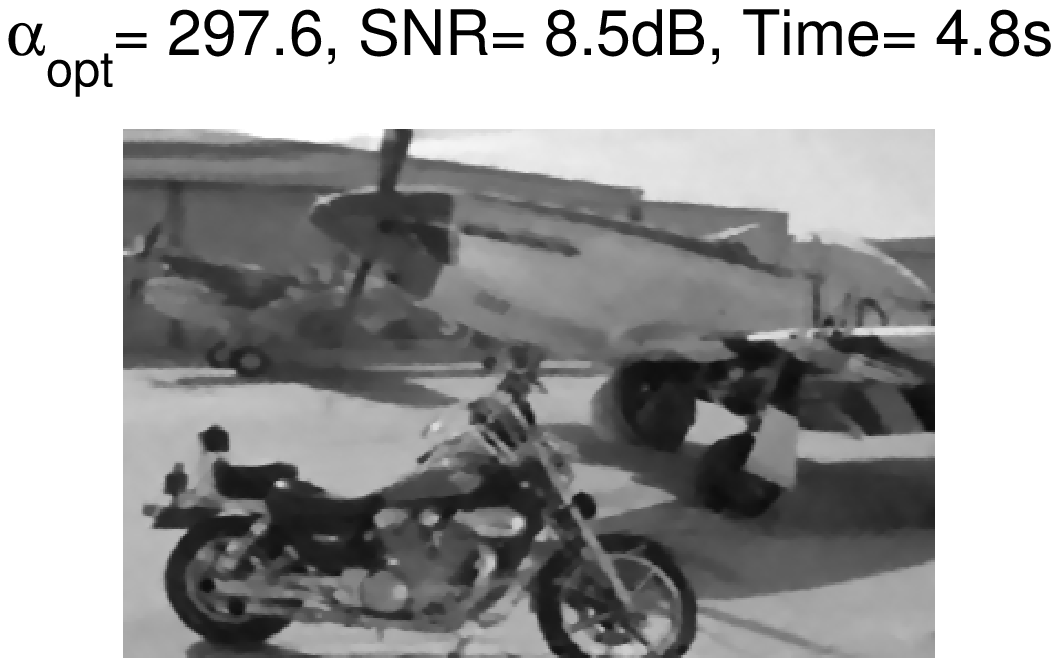,width=5.0cm}
        \small{Symmetric Extension}
   \end{minipage}
    \begin{minipage}[c]{5.0cm}
        \centering
        \epsfig{figure=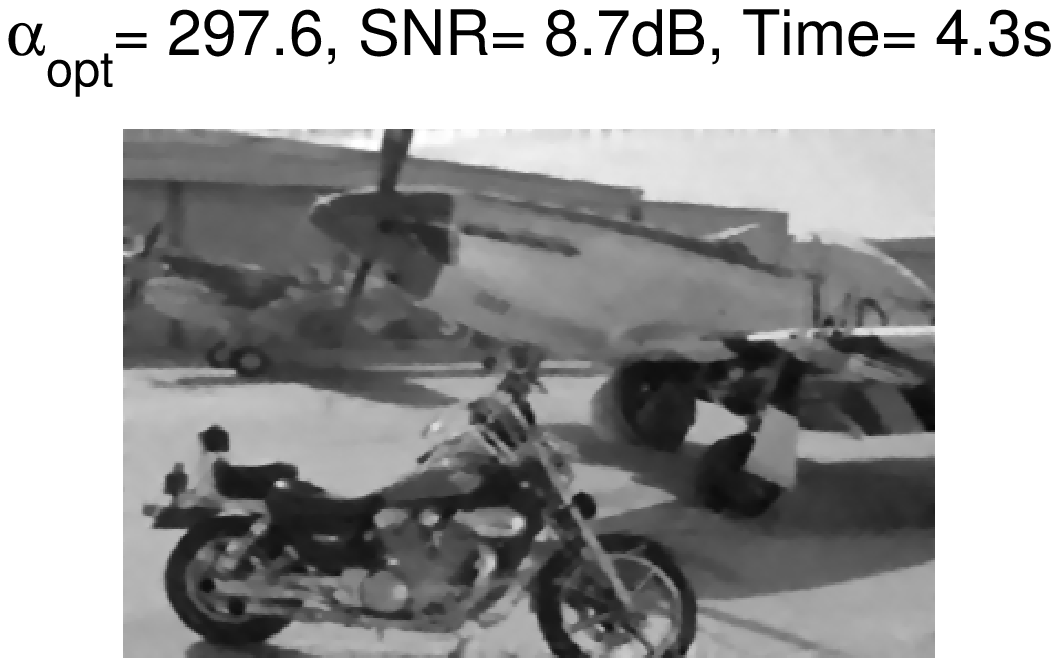,width=5.0cm}
        \small{Anti-symmetric Extension}
    \end{minipage}\\ \vspace{0.2cm}
    \begin{minipage}[c]{5.0cm}
        \centering
        \epsfig{figure=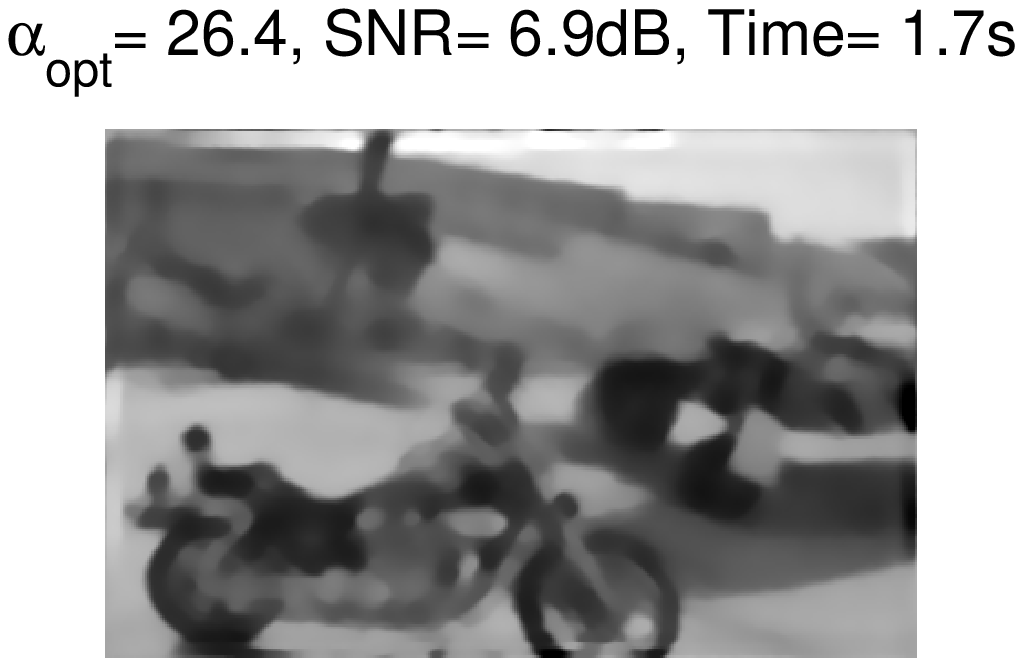,width=5.0cm}
        \small{Periodic}
    \end{minipage}
    \begin{minipage}[c]{5.0cm}
        \centering
        \epsfig{figure=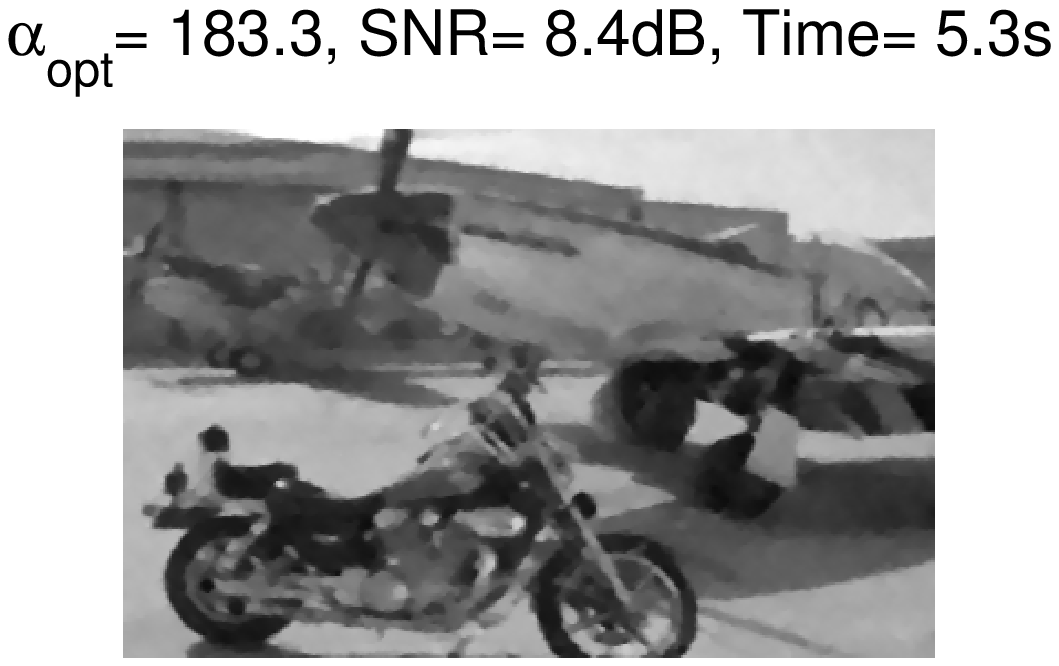,width=5.0cm}
        \small{Symmetric Extension}
   \end{minipage}
    \begin{minipage}[c]{5.0cm}
        \centering
        \epsfig{figure=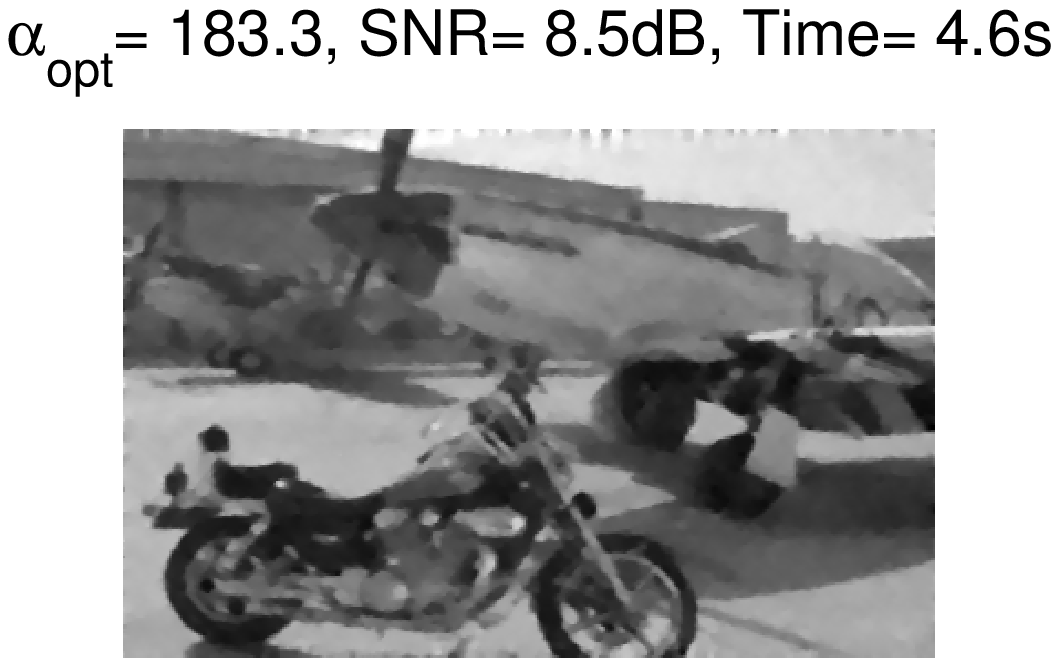,width=5.0cm}
        \small{Anti-symmetric extension}
    \end{minipage}
    \end{center}
\end{center}
\caption{Best restorations for different BCs. The first row contains the results for $\sigma^2=10^{-6}$; the second row contains the results for $\sigma^2=10^{-4}$; the third row contains the results for $\sigma^2=5\times 10^{-4}$.} \label{fig:im-rs5}
\end{figure}

\section{Conclusions}\label{sec5}
In this paper, inspired by \cite{WYYZ08}, we have proposed an alternating minimizzation algorithm for the continuous TV problem \eqref{tv} and we have proved its convergence to the minimum. The continuous formulation allows to solve the minimizzation problem without dealing with the details of the chosen
discretization
strategy. Indeed, we propose two different discrete algorithms
that provide accurate restorations with a low CPU time thanks to the use of fast trigonometric transforms as main computational tool.

In the future, it should be investigated the possibility to discretize our Algorithm \ref{alg} by the strategy 3) illustrated in the Introduction, solving the arising linear system by conjugate gradient with a proper preconditioner. A comparison both in terms of restoration quality and CPU time could be of interest. Furthermore, our approach based on antireflective BCs or antireflective extension could be compared with the algorithms proposed in \cite{AF13,MRF13}. A preliminary comparison is
 given in \cite{AF13} with the reflective extension strategy for their algorithm, but in Section \ref{sec4} we have observed that, at least in the case of symmetric PSF and a low level of noise, the antireflective BCs approach could be a good competitor. This is confirmed also by the numerical results in \cite{VBDW05}, where the antireflective BCs and the strategy 3) are compared using the Landweber method.


\end{document}